\documentclass[namedate,webpdf,imanum]{ima-authoring-template}

\theoremstyle{thmstyletwo}%
\newtheorem{theorem}{Theorem}[section]
\newtheorem{proposition}[theorem]{Proposition}%
\newtheorem{lemma}[theorem]{Lemma}
\newtheorem{corollary}[theorem]{Corollary}

\newtheorem{assumption}[theorem]{Assumption}

\newtheorem{remark}[theorem]{Remark}%

\numberwithin{equation}{section}

\usepackage{mathtools}
\usepackage{bm}

\begin{document}
	
	\DOI{DOI HERE}
	\copyrightyear{2022}
	\vol{00}
	\pubyear{2022}
	\access{Advance Access Publication Date: Day Month Year}
	\appnotes{Paper}
	\copyrightstatement{Published by Oxford University Press on behalf of the Institute of Mathematics and its Applications. All rights reserved.}
	\firstpage{1}
	
	
	\title[Adaptive Planewave Method for Eigenvalue Computations]{Convergence and complexity of an adaptive planewave method for eigenvalue computations}
	
	\author{Xiaoying Dai* and Yan Pan
		\address{\orgdiv{LSEC, Institute of Computational Mathematics and Scientific/Engineering Computing}, \orgname{Academy of Mathematics and Systems Science, Chinese Academy of Sciences}, \orgaddress{\state{Beijing 100190}, \country{China}}}\address{\orgdiv{School of Mathematical Sciences}, \orgname{University of Chinese Academy of Sciences}, \orgaddress{\state{Beijing 100049}, \country{China}}}}
	\author{Bin Yang
		\address{\orgdiv{NCMIS}, \orgname{Academy of Mathematics and Systems Science, Chinese Academy of Sciences}, \orgaddress{\state{Beijing 100190}, \country{China}}}}
	\author{Aihui Zhou
		\address{\orgdiv{LSEC, Institute of Computational Mathematics and Scientific/Engineering Computing}, \orgname{Academy of Mathematics and Systems Science, Chinese Academy of Sciences}, \orgaddress{\state{Beijing 100190}, \country{China}}}\address{\orgdiv{School of Mathematical Sciences}, \orgname{University of Chinese Academy of Sciences}, \orgaddress{\state{Beijing 100049}, \country{China}}}}
	
	\authormark{X. Dai et al.}
	
	\corresp[*]{Corresponding author: \href{email:daixy@lsec.cc.ac.cn}{daixy@lsec.cc.ac.cn}}
	
	\received{Date}{0}{Year}
	\revised{Date}{0}{Year}
	\accepted{Date}{0}{Year}
	
	
	\abstract{In this paper, we study the adaptive planewave discretization for a cluster of eigenvalues of second-order elliptic partial differential equations. We first design an a posteriori error estimator and  prove both the upper  and   lower bounds. Based on the a posteriori error estimator, we propose an adaptive planewave method. We then prove that the adaptive planewave approximations have the linear convergence rate and quasi-optimal complexity.}
	
	
	\maketitle
	
	
	\section{Introduction}
	The mathematical understanding of the adaptive computational methods has derived much attention in mathematical community. We particularly note that the adaptive finite element methods have been extensively investigated for both source problems \citep[see, e.g.,][and references cited therein]{dorfler1996convergent,mekchay2005convergence,stevenson2007optimality,stevenson2008completion,cascon2008quasioptimal,he2011convergencea} and eigenvalue problems \citep[see, e.g.,][and references cited therein]{dai2008convergence,garau2009convergence,giani2009convergent,chen2011adaptive,chen2011finite,chen2014adaptive,garau2011convergence,dai2015convergence,bonito2016convergence}, including the a posteriori error estimates, the convergence and the complexity. The spectral and the pseudospectral methods have been successfully applied in scientific and engineering computation, such as heat conduction, fluid dynamics, quantum physics and so on. For instance, we understand that the planewave discretization methods have been widely used in electronic structure calculations based on the Kohn-Sham equations \citep{kresse1996efficiency,saad2010numerical,chen2013numerical,becke2014perspective}. However, to our best knowledge, there are very few works on adaptive planewave approximations of the partial differential equations. We refer to \citet{gygi1992adaptive} and \citet{liu2022adaptive} for the applications in electronic structure calculations and \citet{canuto2014adaptive,canuto2016adaptive} for the numerical analysis of linear elliptic source problems. There is no any mathematical analysis for adaptive planewave approximations of eigenvalue problems up to now.
	
	In this paper, we first design a residual-type a posteriori error estimator for the planewave approximations of  a class of linear second-order elliptic eigenvalue problems.  We prove that the error estimator can yield both the upper and lower bounds for the error of the approximations. Based on the a posteriori error estimator, we then propose an adaptive planewave method  with the D{\"o}rfler marking strategy \citep{dorfler1996convergent}, which is a typical marking strategy used in adaptive finite element approximations and different from the adaptive planewave method by updating the energy cut-off for planewave discretizations in \citet{liu2022adaptive}. Following \citet{dai2008convergence,dai2015convergence}, by the perturbation arguments, we prove that the adaptive planewave approximations for a cluster of eigenvalues have the asymptotic linear convergence rate and asymptotic quasi-optimal complexity under some reasonable assumptions. More precisely, under the assumption that the initial planewave basis are sufficient enough, we obtain that:
	\begin{itemize}
		\item the associated adaptive planewave approximate eigenspaces $\mathscr{M}_{\mathbb{G}_n}$ will converge to the exact eigenspaces $\mathscr{M}$ with some convergence rate (see Theorem \ref{decreasespace}):
		\begin{equation*}
			\delta_{H_p^1(\Omega)}(\mathscr{M},\mathscr{M}_{\mathbb{G}_n})\lesssim \alpha^n,
		\end{equation*}
		where $\alpha\in (0,1)$ is some constant.
		\item if $M(\lambda_{(i)})\subset\mathcal{A}^s$ for $i=1,2,\ldots,m$ and the marked indexes are of minimal cardinality, the adaptive planewave approximations have a quasi-optimal complexity as follows (see Theorem \ref{thm:opt-complex}):
		\begin{equation*}
			\delta_{H_p^1(\Omega)}(\mathscr{M},\mathscr{M}_{\mathbb{G}_n}) \lesssim (|\mathbb{G}_n|-|\mathbb{G}_0|)^{-s}.
		\end{equation*}
	\end{itemize}
	We refer to Section \ref{sec:conv-complex} for more details.
	
	The rest of this paper is organized as follows. In Section \ref{sec:pre}, we describe some basic notation and review the existing results of planewave approximations for a class of linear second-order elliptic source and eigenvalue problems that will be useful in our analysis. In Section \ref{sec:APWM}, we present a posteriori error estimators from the relationship between the elliptic eigenvalue approximations with the associated source approximations. We then design an adaptive planewave method and its feasible version for an elliptic eigenvalue problem. In Section \ref{sec:conv-complex}, we analyze the asymptotic convergence and asymptotic quasi-optimal complexity of the adaptive planewave method. Finally, some conclusion remarks are given in Section \ref{sec:concl}.

	\section{Preliminaries}\label{sec:pre}
	Let $\Omega=[0,2\pi)^d(d\ge 1)$. Denote the family of periodic continuous functions by
	\[
	C_p^0(\Omega)=\{ v\in C^0(\mathbb{R}^d):v(\bm{x}+2\pi \bm{n})=v(\bm{x}),~1\le j\le d,~\forall \bm{n}\in\mathbb{Z}^d,~\forall \bm{x}\in\Omega \}.
	\]
	For any $\bm{G}\in\mathbb{Z}^d$, we denote $e_{\bm{G}}(\bm{x})=1/(2\pi)^{d/2}e^{\mathrm{i}\bm{G}\cdot \bm{x}}$, $\bm{x}\in\mathbb{R}^d$. It is clear that the family $\{e_{\bm{G}}\}_{\bm{G}\in\mathbb{Z}^d}$ forms an orthonormal basis of
	\[
	L_p^2(\Omega)=\{ v\in L_{\text{loc}}^2(\mathbb{R}^d): v(\bm{x}+2\pi \bm{n})=v(\bm{x}),~1\le j\le d,~\forall \bm{n}\in\mathbb{Z}^d,~\forall \bm{x}\in\Omega \},
	\]
	and for any $v\in L_p^2(\Omega)$,
	\[
	v = \sum_{\bm{G}\in\mathbb{Z}^d}\hat{v}_{\bm{G}}e_{\bm{G}}\quad\text{with}\quad\hat{v}_{\bm{G}}=(e_{\bm{G}},v)=\frac{1}{(2\pi)^{d/2}}\int_\Omega v(\bm{x})e^{-\mathrm{i}\bm{G}\cdot \bm{x}}\operatorname{d\!}\bm{x}\textup{d}\bm{x}.
	\]
	Note that $\|v\|_{L^2(\Omega)}^2=\sum\limits_{\bm{G}\in\mathbb{Z}^d}|\hat{v}_{\bm{G}}|^2$.
	We shall use the notation of Sobolev spaces $H_p^s(\Omega)$ with $s\in\mathbb{R}$ for real valued periodic functions
	\[
	H_p^s(\Omega)=\left\{u=\sum_{\bm{G}\in\mathbb{Z}^d}\hat{u}_{\bm{G}} e_{\bm{G}}: \sum_{\bm{G}\in\mathbb{Z}^d}(1+|\bm{G}|^2)^s|\hat{u}_{\bm{G}}|^2<\infty~\text{and}~\forall \bm{G}\in\mathbb{Z}^d,~\hat{u}_{-\bm{G}}=\hat{u}_{\bm{G}}^*\right\}
	\]
	endowed with the inner products
	\[
	(u,v)_{H_p^s(\Omega)}=\sum_{\bm{G}\in\mathbb{Z}^d}(1+|\bm{G}|^2)^s\hat{u}_{\bm{G}}^*\hat{v}_{\bm{G}}.
	\]
	Here and hereafter, $|\bm{G}|$ denotes the Euclidean norm of the multi-index $\bm{G}$ and we will omit the domain $\Omega$ when we express norm and inner product. Throughout this paper, $A\lesssim B$ means $A\le CB$ with some generic positive constant $C$ that is independent of the index set $\mathbb{G}\subset\mathbb{Z}^d$  and $A\cong B$ means $B\lesssim A\lesssim B$.
	
	Given any finite index set $\mathbb{G}\subset\mathbb{Z}^d$ satisfying $\mathbb{G}=-\mathbb{G}$, we define the subspace of $H_p^1(\Omega)$
	\[
	V_\mathbb{G}\coloneqq\operatorname{span}\{ e_{\bm{G}}:\bm{G}\in\mathbb{G} \}\cap H_p^1(\Omega).
	\]
	We set $|\mathbb{G}|=$ the cardinality of $\mathbb{G}$. For any positive integer $M$, specially, we define $\mathbb{G}^{M}$  by
	\[
	\mathbb{G}^{M}=\{ \bm{G}\in\mathbb{Z}^d:|\bm{G}|\le M  \}.
	\]
	Throughout this paper, if not specially specified, $\mathbb{G}\subset\mathbb{Z}^d$ always satisfies $\mathbb{G}=-\mathbb{G}$ since only real valued periodic functions are taken into account. We shall also define the $L^2$-projection $\Pi_\mathbb{G}:L^2_p(\Omega)\rightarrow V_\mathbb{G}$ by
	\[
	(u-\Pi_\mathbb{G} u,v)=0,~\forall  v\in V_\mathbb{G}.
	\]
	The following results can be found in \citet{cances2010numerical}.
	\begin{proposition}\label{H1L2error}
		If $u\in H_p^m(\Omega)$ for some $m\ge 0$, then
		\begin{equation*}
			\|u-\Pi_{\mathbb{G}^{M}} u\|_{H_p^l}\le \frac{1}{M^{m-l}}\|u\|_{H_p^{m}},\quad 0\le l\le m.
		\end{equation*}
\end{proposition}

\subsection{Planewave approximation of a source problem}
We consider the following source system:
\begin{equation}
	\left\{\begin{aligned}\label{bvp}
		& Lu_i= f_i~~\textrm{in} \quad\Omega,\quad i=1,\dots,N,\\
		& u_i(\bm{x}+2\pi \bm{n})=u_i(\bm{x})\quad \forall \bm{n}\in\mathbb{Z}^d.
	\end{aligned}
	\right.
\end{equation}
where $L$ is a linear second order elliptic operator:
\begin{equation*}
	Lu= -\Delta u + V u.
\end{equation*}
We assume $V\in H_p^\sigma(\Omega), \sigma> d/2$ and $V>0$. Since $H_p^\sigma(\Omega)\hookrightarrow C_p^0(\Omega)$, there exist $0<\nu_*<\nu^*<\infty$ such that $\nu_*\le V\le\nu^*$. Note that $u_i\in H_p^2(\Omega)$ provided $f_i\in L^2(\Omega)$ \citep{cances2010numerical}. It should be mentioned that the assumption $V>0$ is just for brevity \citep[see, e.g.,][Remark 2.9]{dai2008convergence}.

Let 
\[
a(u,v)=(Lu,v)=(\nabla u,\nabla v)+(Vu,v),~\forall u,v\in H_p^1(\Omega).
\]
We define by $\|v\|_a=\sqrt{a(v,v)}$ the energy norm of any $v\in H_p^1(\Omega)$, which satisfies
\begin{equation*}
	\sqrt{\alpha_*}\|v\|_{H_p^1}\leq\|v\|_{a}\leq\sqrt{\alpha^*}\|v\|_{H_p^1},
\end{equation*}
where $\alpha_*=\min(\nu_*,1),\alpha^*=\max(\nu^*,1)$. Define operator $K:L_p^2(\Omega)\rightarrow H^1_p(\Omega)$ as follows:
\begin{equation}\label{defoperaK}
	a(Kw,v) = (w,v),\quad\forall w\in L_p^2(\Omega),~~\forall v\in H^1_p(\Omega).
\end{equation}
We see that $K$ is well-defined and compact. Define the Galerkin-projection $P_\mathbb{G}:~H_p^1(\Omega)\mapsto V_\mathbb{G}$ by
\begin{equation}\label{galerpro}
	a(u-P_\mathbb{G} u,v)=0,\quad\forall v\in V_\mathbb{G},
\end{equation}
and the operator $K_\mathbb{G}:L_p^2(\Omega)\mapsto V_\mathbb{G}$ by
\begin{equation*}
	a(K_\mathbb{G} w,v) = (w,v),\quad\forall w\in L_p^2(\Omega),~~\forall v\in V_\mathbb{G}.
\end{equation*}
We have $K_\mathbb{G} = P_\mathbb{G} K$.

The following conclusion can be found in \citet{babuska1989finite} and \citet{xu2000local}.
\begin{lemma}\label{rholimits0}
	Let
	\begin{equation*}
		\rho_\Omega(\mathbb{G})=\sup\limits_{f\in L^2(\Omega),\|f\|_{L^2}=1}\inf\limits_{v\in V_\mathbb{G}}\|Kf-v\|_{a}.
	\end{equation*}
	Then 
	\begin{equation*}
		\|u-P_\mathbb{G} u\|_{L^2_p}\le \rho_\Omega(\mathbb{G})\|u-P_\mathbb{G} u\|_a,\quad\forall u\in H_p^1(\Omega)
	\end{equation*}
	and $\rho_{\Omega}(\mathbb{G}^{M})\rightarrow 0$ as $M\rightarrow \infty$.
\end{lemma}

Note that $\rho_\Omega(\mathbb{G})\ge \rho_\Omega(\mathbb{G}^*)$ if $\mathbb{G}\subset\mathbb{G}^*$. 

For any $U=(u_1,\dots,u_N)$, we denote
\begin{equation*}
	\|U\|=\left(\sum_{i=1}^N\|u_i\|^2\right)^{1/2}
\end{equation*}
for any relevant norm $\|\cdot\|$.

The weak form of problem \eqref{bvp} reads as follows: find $U= (u_1,\dots,u_N)\in (H^1_p(\Omega))^N$ such that
\begin{equation}\label{weakformbvp}
	a(u_i,v_i) = (f_i,v_i)~~\forall v_i\in H^1_p(\Omega),~i=1,\dots,N.
\end{equation}

Given a finite index set $\mathbb{G}\subset \mathbb{Z}^d$ satisfying $\mathbb{G}=-\mathbb{G}$, the planewave approximation of \eqref{weakformbvp} in $V_\mathbb{G}$ reads: find
$ U_\mathbb{G}=(u_{1,\mathbb{G}},\dots,u_{N,\mathbb{G}})\in (V_\mathbb{G})^N$ such that
\begin{equation}\label{discreformbvp}
	~~a(u_{i,\mathbb{G}},v_{i,\mathbb{G}}) = (f_i,v_{i,\mathbb{G}}),\quad\forall v_{i,\mathbb{G}}\in V_\mathbb{G},~i=1,\dots,N.
\end{equation}

For any $w_{i,\mathbb{G}}\in V_\mathbb{G}$, the residual $\bar{r}(w_{i,\mathbb{G}})$ is defined by $\bar{r}(w_{i,\mathbb{G}})=f_i-Lw_{i,\mathbb{G}}$ and the a posterior error estimator $\bar{\eta}(w_{i,\mathbb{G}};\mathbb{G}^*)$ for any $\mathbb{G}^*\subset\mathbb{Z}^d$ is defined by
\[
\bar{\eta}^2(w_{i,\mathbb{G}};\mathbb{G}^*)=\|\Pi_{\mathbb{G}^*}\bar{r}(w_{i,\mathbb{G}})\|_{H_p^{-1}(\Omega)}^2=\sum_{\bm{G}\in \mathbb{G}^*}\frac{|\hat{\bar{r}}_{\bm{G}}(w_{i,\mathbb{G}})|}{1+|\bm{G}|^2},
\]
which is used in \citet{canuto2014adaptive}. For any $W_\mathbb{G}=(w_{1,\mathbb{G}},\ldots,w_{N,\mathbb{G}})\in (V_\mathbb{G})^N$ and $\mathbb{G}^*\subset\mathbb{Z}^d$, let 
\begin{equation*}
	\bar{r}(W_{\mathbb{G}}) = (f_1-Lw_{1,\mathbb{G}},\ldots,f_N-Lw_{N,\mathbb{G}})\quad\text{and}\quad\bar{\eta}^2(W_\mathbb{G};\mathbb{G}^*)=\sum_{i=1}^N\bar{\eta}^2(w_{i,\mathbb{G}};\mathbb{G}^*).
\end{equation*}
We shall abbreviate $\bar{\eta}(w_{\mathbb{G}};\mathbb{Z}^d)$ and $\bar{\eta}^2(W_\mathbb{G};\mathbb{Z}^d)$ to $\bar{\eta}(w_{\mathbb{G}})$ and $\bar{\eta}^2(W_\mathbb{G})$, respectively.

The following proposition is a direct result from (2.9) of \citet{canuto2014adaptive}.
\begin{proposition}
	If $U=(u_1,u_2,\ldots,u_N)$ is the solution of \eqref{bvp}, then
	\begin{equation*}
		\frac{1}{\sqrt{\alpha^*}}\|\bar{r}(W_{\mathbb{G}})\|_{H_p^{-1}}\leq\|U-W_{\mathbb{G}}\|_a\leq\frac{1}{\sqrt{\alpha_*}}\|\bar{r}(W_{\mathbb{G}})\|_{H_p^{-1}},\quad \forall W_{\mathbb{G}}\in (V_\mathbb{G})^N,
	\end{equation*}
	or namely,
	\begin{equation}\label{estimatbvp}
		\frac{1}{\sqrt{a^*}}\bar{\eta}(W_{\mathbb{G}})\leq
		\|U-W_{\mathbb{G}}\|_a\leq\frac{1}{\sqrt{a_*}}\bar{\eta}(W_{\mathbb{G}}),\quad \forall W_{\mathbb{G}}\in (V_\mathbb{G})^N.
	\end{equation}	
\end{proposition}

The D\"{o}rfler marking strategy \citep{dorfler1996convergent,canuto2014adaptive}, which will be applied in our adaptive planewave method, is stated as the following general form:
\begin{algorithm}[H]
	\caption{\textbf{D\"orfler marking strategy:} $\operatorname{\text{D\"{O}RFLER}}(\eta,U,\mathbb{G},\theta)$}
	\begin{algorithmic}[1]
		\State Construct a subset $\delta \mathbb{G}\subset\mathbb{G}^c\coloneqq\mathbb{Z}^d\setminus \mathbb{G}$ satisfying $\delta \mathbb{G}=-\delta \mathbb{G}$ and
		\begin{equation*}
			\eta(U;\delta \mathbb{G})\geq\theta\eta(U);
		\end{equation*}
		\State Return $\delta \mathbb{G}$.
	\end{algorithmic}
\end{algorithm}

The adaptive planewave algorithm with D\"{o}rfler marking strategy for solving \eqref{weakformbvp} is stated as follows \citep[cf.][]{canuto2014adaptive}:

\begin{algorithm}[H]
	\caption{Adaptive planewave algorithm for the source problem}
	\label{algo:APWM}
	\begin{algorithmic}[1]
		\State Choose parameters $\theta\in(0,1)$ and $tol\in[0,1)$;
		\State Set $\bar{r}(u_{i,\mathbb{G}_0})=f_i\,(i=1,2,\ldots,N)$, $\mathbb{G}_0=\emptyset$ and $n=0$;
		\State\label{algostate:compute-eta} Compute the error estimator $\bar{\eta}(U_{\mathbb{G}_{n}})$;
		\State If $\bar{\eta}(U_{\mathbb{G}_{n}}) < tol$, then stop;
		\State Construct $\delta \mathbb{G}_n=\operatorname{\text{D\"{O}RFLER}}(\bar{\eta},U_{\mathbb{G}_n},\mathbb{G}_n,\theta)$;
		\State Let $\mathbb{G}_{n+1}=\mathbb{G}_n\cup\delta\mathbb{G}_n$;
		\State Solve \eqref{discreformbvp} with $\mathbb{G}$ being replaced by $\mathbb{G}_{n+1}$ to get the discrete solution $U_{\mathbb{G}_{n+1}}=(u_{1,\mathbb{G}_{n+1}},\dots,u_{N,\mathbb{G}_{n+1}})\in (V_{\mathbb{G}_{n+1}})^N$;
		\State Let $n=n+1$ and go to Step \ref{algostate:compute-eta}.
	\end{algorithmic}
\end{algorithm}

We observe that $\bar{\eta}(U_{\mathbb{G}})$ is an infinite sum over $\mathbb{Z}^d$ and Algorithm \ref{algo:APWM} is indeed not practicable.

The following conclusion is a direct extension of Theorem 3.1 in \citet{canuto2014adaptive} from the case of $N=1$ to the case of any $N$.
\begin{theorem}
	Let $\rho=\rho(\theta)=\sqrt{1-\frac{\alpha_*}{\alpha^*}\theta^2}\in(0,1)$ and $\{U_{\mathbb{G}_n}\}_{n\geq0}$ be the sequence generated by Algorithm \ref{algo:APWM}, then
	\begin{equation}\label{convidealrho}
		\|U-U_{\mathbb{G}_{n+1}}\|_a\leq\rho\|U-U_{\mathbb{G}_n}\|_a,\quad n=1,2,\ldots
	\end{equation}
\end{theorem}

\subsection{A linear eigenvalue problem}
Consider the following elliptic eigenvalue problem:
\begin{equation}
	\left\{\begin{aligned}\label{evp}
		& Lu = \lambda u~~\textrm{in} ~~\Omega,\\
		& u(\bm{x}+2\pi \bm{n})=u(\bm{x})~~\forall \bm{n}\in\mathbb{Z}^d.
	\end{aligned}
	\right.
\end{equation}

The weak form of problem \eqref{evp} reads: find $(\lambda,u)\in\mathbb{R}\times H_p^1(\Omega)$ such that
\begin{equation}\label{weakform}
	a(u,v)=\lambda(u,v),\quad\forall v\in H_p^1(\Omega).
\end{equation}

We see that \eqref{weakform} has a sequence of real eigenvalues
$$0<\lambda_1<\lambda_2\leq\lambda_3\leq\cdots$$
and the corresponding eigenfunctions
$$u_1,u_2,u_3,\ldots,$$
which can be assumed to satisfy
$$(u_i,u_j)=\delta_{ij},~i,j=1,2,\ldots,$$
where the $\lambda_j$s are repeated according to multiplicity.

Given any finite index set $\mathbb{G}\subset\mathbb{Z}^d$ satisfying $\mathbb{G}=-\mathbb{G}$, the planewave approximation of \eqref{weakform} in $V_\mathbb{G}$ reads: find $u_\mathbb{G}\in V_\mathbb{G}$ such that
\begin{equation}\label{discreform}
	a(u_\mathbb{G},v_\mathbb{G}) = \lambda_\mathbb{G}(u_\mathbb{G},v_\mathbb{G}),\quad\forall v_\mathbb{G}\in V_\mathbb{G}.
\end{equation}
The eigenvalues of \eqref{discreform} can be ordered as follows:
$$0<\lambda_{1,\mathbb{G}}\le\lambda_{2,\mathbb{G}}\leq\cdots\le\lambda_{n_\mathbb{G},\mathbb{G}},\quad n_\mathbb{G}=\textrm{dim} V_\mathbb{G},$$
and corresponding eigenfunctions may be denoted by
$$u_{1,\mathbb{G}},u_{2,\mathbb{G}},\ldots,u_{n_\mathbb{G},\mathbb{G}},$$
satisfying
$$(u_{i,\mathbb{G}},u_{j,\mathbb{G}})=\delta_{ij},i,j=1,2,\ldots,n_\mathbb{G}.$$
By the definition of operators $K$ and $K_\mathbb{G}$, it is clear that $K$ has eigenvalues 
\[
\lambda_1^{-1}>\lambda_2^{-1}\ge\lambda_3^{-1}\ge\cdots
\]
associated with eigenfunctions $u_1,u_2,u_3,\ldots$, and $K_\mathbb{G}$ has eigenvalues
\[
\lambda_{1,\mathbb{G}}^{-1}\ge\lambda_{2,\mathbb{G}}^{-1}\ge\cdots\ge\lambda_{n_\mathbb{G},\mathbb{G}}^{-1}
\]
associated with eigenfunctions $u_{1,\mathbb{G}},u_{2,\mathbb{G}},\ldots,u_{n_\mathbb{G},\mathbb{G}}$.

Let $\lambda$ be some eigenvalue of \eqref{weakform}. We set
\[
M(\lambda)=\{w\in H_p^1(\Omega):w\mbox{ is an eigenfunction of \eqref{weakform} corresponding to } \lambda\}.
\]
Let $\Gamma$ be a circle in the complex plane centered at $\lambda^{-1}$ and not enclosing any other eigenvalue of $K$. Define the spectral projection $E: L_p^2(\Omega)\to M(\lambda)$ associated with $K$ and $\lambda$ by
\[
E=E(\lambda)=\frac{1}{2\pi\mathrm{i}}\int_{\Gamma}(z-K)^{-1}dz.
\]
When $\mathbb{G}\supset \mathbb{G}^{M}$ and $M$ is large enough, we can define the spectral projection associated with $K_\mathbb{G}$ as \[
E_{\mathbb{G}}=E_\mathbb{G}(\lambda)=\frac{1}{2\pi\mathrm{i}}\int_{\Gamma}(z-K_\mathbb{G})^{-1}dz.
\]

\begin{lemma}\label{deltalimits0}
	Let 
	\begin{equation*}
		\delta_\mathbb{G}(\lambda)=\sup\limits_{w\in M(\lambda),\|w\|_{a}=1}\inf\limits_{v\in V_\mathbb{G}}\|w-v\|_{a,\Omega}.
	\end{equation*}
	Then $\delta_{\mathbb{G}^{M}}(\lambda)\rightarrow 0$ as $M\rightarrow \infty$.
\end{lemma}
\begin{proof}
	For any $w\in M(\lambda)$ with $\|w\|_{a}=1$, we have
	\[
	\inf_{v\in V_{\mathbb{G}^{M}}}\|w-v\|_a\le \|w-\Pi_{\mathbb{G}^{M}} w\|_a.
	\]
	We see from $w=\lambda Kw$ that $w\in H_p^2(\Omega)$. Hence Proposition \ref{H1L2error} implies
	\[
	\|w-\Pi_{\mathbb{G}^{M}}w\|_a\le\sqrt{\alpha^*}\|w-\Pi_{\mathbb{G}^{M}}w\|_{H_p^1}\le\frac{\sqrt{\alpha^*}}{M}\|w\|_{H_p^2},
	\]
	which together with $K: H_p^1(\Omega)\to H_p^2(\Omega)$ being bounded completes the proof.
\end{proof}

We note that $\delta_\mathbb{G}(\lambda)\ge \delta_{\mathbb{G}^*}(\lambda)$ if $\mathbb{G}\subset\mathbb{G}^*$.

Similar to \citet{babuska1989finite,babuska1991eigenvalue}, we have the following propositions.
\begin{proposition}
	Let $\mathbb{G}\supset \mathbb{G}^M$, $\lambda\in\mathbb{R}$ be any eigenvalue of \eqref{weakform} with multiplicity $q$, and $\lambda_{\mathbb{G},l}(l=1,2,\ldots,q)$ be the eigenvalues of \eqref{discreform} which approximate $\lambda$. Then
	\begin{align}
		&\|u-E_\mathbb{G} u\|_{L^2}\lesssim \rho_{\Omega}(\mathbb{G})\|u-E_\mathbb{G} u\|_a, ~\|u-E_\mathbb{G} u\|_a\lesssim\delta_\mathbb{G}(\lambda)\|u\|_a~~\forall u\in M(\lambda), \\
		&\lambda_{\mathbb{G},l}-\lambda\lesssim \delta_\mathbb{G}(\lambda)^2 \label{lambdaineq}
	\end{align}
	provided $M\gg 1$.
\end{proposition}

\begin{proposition}\label{euandpu}
	Let $\mathbb{G}\supset \mathbb{G}^M$. For any $u\in M(\lambda)$, there holds
	\begin{equation*}
		1\leq \frac{\|u-E_{\mathbb{G}} u\|_a}{\|u-P_\mathbb{G} u\|_a}=1+\mathcal {O}(\nu(\mathbb{G}))
	\end{equation*}
	provided $M\gg 1$, where $\nu(\mathbb{G})=\sup\limits_{f\in H_p^1(\Omega),\|f\|_a=1}\inf\limits_{v\in V_\mathbb{G}}\|Kf-v\|_a$.
\end{proposition}

Similar to Lemma \ref{deltalimits0}, we have the following proposition.
\begin{proposition}
	$\nu({\mathbb{G}^{M}})\rightarrow 0$ as $M\rightarrow\infty$.	
\end{proposition}
\begin{proof}
	For any $f\in H_p^1(\Omega)$ with $\|f\|_a=1$, we have
	\[
	\inf_{v\in V_{\mathbb{G}^{M}}}\|Kf-v\|_a\le \|Kf-\Pi_{\mathbb{G}^{M}}Kf\|_a.
	\]
	Set $w=Kf$, we get
	\[
	\inf_{v\in V_{\mathbb{G}^{M}}}\|Kf-v\|_a\le \|w-\Pi_{\mathbb{G}^{M}}w\|_a.
	\]
	Thanks to $w\in H_p^2(\Omega)$ and Proposition \ref{H1L2error}, we obtain
	\[
	\|w-\Pi_{\mathbb{G}^{M}}w\|_a\le\sqrt{\alpha^*}\|w-\Pi_{\mathbb{G}^{M}}w\|_{H^1_p}\le\frac{\sqrt{\alpha^*}}{M}\|w\|_{H_p^2}.
	\]
	Since $K:H_p^1(\Omega)\rightarrow H_p^2(\Omega)$ is bounded, we arrive at the conclusion.
\end{proof}

Following \citet{dai2015convergence}, we obtain the following results.
\begin{proposition}\label{elambdaul2}
	Let $\mathbb{G}\supset \mathbb{G}^M$. For any $u\in M(\lambda)$ with $\|u\|_{L^2}=1$, there holds
	\begin{equation*}
		1-C\rho_\Omega(\mathbb{G})\delta_\mathbb{G}(\lambda)\leq\|E_\mathbb{G} u\|^2_{L^2}\leq 1
	\end{equation*}
	provided $M\gg 1$, where $C$ is some constant not depending on $\mathbb{G}$.
\end{proposition}

\begin{proposition}\label{elamdaiandj}
	Let $\mathbb{G}\supset \mathbb{G}^M$. For any $u_i,u_j\in M(\lambda)$ with $(u_i,u_j)=\delta_{ij}$, there holds
	\begin{equation*}
		(E_\mathbb{G} u_i,E_\mathbb{G} u_j) = \delta_{ij}+\mathcal {O}(\rho_\Omega(\mathbb{G})\delta_\mathbb{G}(\lambda))
	\end{equation*}
	provided $M\gg 1$.
\end{proposition}

\section{Adaptive planewave method}\label{sec:APWM}
Here and hereafter, we consider the planewave approximation for a cluster of eigenvalues of \eqref{weakform}. For a cluster of eigenvalues $\lambda_{k_0+1}\le\cdots\le\lambda_{k_0+N}$ of \eqref{weakform}, we assume
\[
\lambda_{k_0}<\lambda_{k_0+1}\le\cdots\le\lambda_{k_0+N}<\lambda_{k_0+N+1}.
\]
Here, we take $\lambda_0=0$ if $k_0=0$. If not accounting the multiplicity, we assume that the $N$ eigenvalues $\lambda_{k_0+1},\ldots,\lambda_{k_0+N}$ belong to $m$ eigenvalues $\lambda_{(1)}<\lambda_{(2)}<\ldots<\lambda_{(m)}$ with the multiplicity of each eigenvalue being $q_i$. Thus $N=\sum_{i=1}^m q_i$.

Let $\lambda_{\mathbb{G},l}$ be the $(k_0+l)$th eigenvalue of \eqref{discreform} and $u_{\mathbb{G},l}$ be the eigenfunction corresponding to $\lambda_{\mathbb{G},l}(l=1,\ldots,N)$, respectively.
Set
$$M_\mathbb{G}(\lambda_{(i)})=\operatorname{span}\{u_{\mathbb{G},k_i+1},u_{\mathbb{G},k_i+2},\ldots,u_{\mathbb{G},k_i+q_i}\},$$
where $k_i = \sum_{j=1}^{i-1}q_i,\,i=1,2,\ldots,m$.

Let $\Gamma_i$ be a circle in the complex plane centered at $\lambda_{(i)}^{-1}$ and not enclosing any other eigenvalue of $K$. Set 
\[
E_{\mathbb{G},i}=E_\mathbb{G}(\lambda_{(i)})=\frac{1}{2\pi\mathrm{i}}\int_{\Gamma_i}(z-K_\mathbb{G})^{-1}dz.
\]
We see that $E_\mathbb{G}(\lambda_{(i)}):M(\lambda_{(i)})\to M_\mathbb{G}(\lambda_{(i)})$ is a bijection when $\mathbb{G}\supset \mathbb{G}^M$ and $M$ is large enough \citep[see, e.g.,][]{babuska1989finite,babuska1991eigenvalue}. Define
\[
\mathscr{M}=\begin{pmatrix}
	M(\lambda_{(1)})\\
	M(\lambda_{(2)})\\
	\vdots\\
	M(\lambda_{(m)})
\end{pmatrix},\quad
\mathscr{M}_{\mathbb{G}}=\begin{pmatrix}
	M_{\mathbb{G}}(\lambda_{(1)})\\
	M_\mathbb{G}(\lambda_{(2)})\\
	\vdots\\
	M_\mathbb{G}(\lambda_{(m)})
\end{pmatrix}
\]
and consider the following operator:
\[
\mathscr{E}_{\mathbb{G}}=\oplus_{i=1}^m E_{\mathbb{G}}(\lambda_{(i)}):\mathscr{M}\to \mathscr{M}_\mathbb{G}.
\]

For any $u\in M(\lambda_{(i)})$ with $\|u\|_{L^2}=1$, we see from $E_{\mathbb{G},i} u\in M_\mathbb{G}(\lambda_{(i)})$ that there exist some constants $\{\alpha^\mathbb{G}_{i,l}(u)\}_{l=1}^{q_i}$ such that
$$E_{\mathbb{G},i} u=\sum_{l=1}^{q_i} \alpha^\mathbb{G}_{i,l}(u) u_{\mathbb{G},k_i+l}.$$
Note that \eqref{weakform} and \eqref{discreform} can be rewritten as
\begin{equation*}
	u= \lambda Ku,\quad u_\mathbb{G}= \lambda_\mathbb{G} P_\mathbb{G} K u_\mathbb{G},
\end{equation*}
where $K$ and $P_\mathbb{G}$ are defined by \eqref{defoperaK} and
\eqref{galerpro}, respectively. Let 
\[
w^{\mathbb{G},i}=\sum_{l=1}^{q_i}\alpha^\mathbb{G}_{i,l}(u) \lambda_{\mathbb{G},k_i+l}Ku_{\mathbb{G},k_i+l},
\]
then
\begin{equation}\label{eq:Elambda_u}
	E_{\mathbb{G},i} u = P_\mathbb{G} w^{\mathbb{G},i}.
\end{equation}

Using the similar arguments in the proof of Theorem 3.1 of \citet{dai2015convergence}, we have the following theorem.

\begin{theorem}\label{thm:u-Eu-appro-w-Pw}
	Let $\mathbb{G}\supset \mathbb{G}^M$ and $\displaystyle \gamma_\mathbb{G}:=\rho_{\Omega}(\mathbb{G})+\max_{i=1,\ldots,m}\delta_\mathbb{G}(\lambda_{(i)})$, then for any $u\in M(\lambda_i)(i=1,2,\ldots,m)$ with $\|u\|_{L^2}=1$, there holds
	\begin{equation}\label{eigenbvp}
		\|u-E_{\mathbb{G},i} u\|_a=\|w^{\mathbb{G},i}-P_\mathbb{G} w^{\mathbb{G},i}\|_a+\mathcal {O}(\gamma_\mathbb{G})\|u-E_{\mathbb{G},i} u\|_a
	\end{equation}
	provided $M\gg 1$.
\end{theorem}

Note that $u-E_{\mathbb{G},i} u=w^{\mathbb{G},i}-P_\mathbb{G} w^{\mathbb{G},i}+u-w^{\mathbb{G},i}$ and 
\begin{equation}\label{uminusw}
	\|u-w^{\mathbb{G},i}\|_a\leq \tilde{C}\gamma_\mathbb{G}\|u-E_{\mathbb{G},i} u\|_a,\quad i=1,2,\ldots,m
\end{equation}
are applied in the proof of Theorem \ref{thm:u-Eu-appro-w-Pw}.

\begin{remark}
	Lemmas \ref{rholimits0} and \ref{deltalimits0} imply that $\gamma_{\mathbb{G}^{M}}\rightarrow 0$ as $M\rightarrow \infty$.
\end{remark}

To evaluate the error of the approximation, we introduce a distance between subspaces $X$ and $Y$ of $H_p^1(\Omega)$ as follows
\begin{equation*}
	\delta_{H_p^1(\Omega)}(X,Y):=\max\{ d_{H_p^1(\Omega)}(X,Y),d_{H_p^1(\Omega)}(Y,X)\},
\end{equation*}
where
\begin{equation*}
	d_{H_p^1(\Omega)}(X,Y):=\sup\limits_{u\in X,\|u\|_{a}=1}\inf\limits_{v\in Y}\|u-v\|_{a}.
\end{equation*}

For $d_{H_p^1(\Omega)}(X,Y)$ defined above, the following lemma holds \citep[see, e.g.,][]{babuska1991eigenvalue}.
\begin{lemma}\label{dxydyx}
	If $\dim X=\dim Y<\infty$, then $$d_{H_p^1(\Omega)}(Y,X)\leq d_{H_p^1(\Omega)}(X,Y)[1-d_{H_p^1(\Omega)}(X,Y)]^{-1}.$$
\end{lemma}

\subsection{A posteriori error estimators}
For any eigenpair approximation $(\lambda_{\mathbb{G},l},u_{\mathbb{G},l})~(l=1,2,\ldots,N)$, we define the residual $r(u_{\mathbb{G},l})$ by
\[
	r(u_{\mathbb{G},l})= \lambda_{\mathbb{G},l} u_{\mathbb{G},l}+\Delta u_{\mathbb{G},l}-V u_{\mathbb{G},l},\\
\]
Instead, for any $E_{\mathbb{G},i}u$ with $u\in M(\lambda_{(i)})$, we define the residual $r(E_{\mathbb{G},i} u)$ by
\[
r(E_{\mathbb{G},i} u)=\sum_{l=1}^{q_i}\alpha^\mathbb{G}_{i,l}(u)\lambda_{\mathbb{G},k_i+l}u_{\mathbb{G},k_i+l}+\Delta E_{\mathbb{G},i} u-V E_{\mathbb{G},i} u,
\]
The associated error estimators $\eta(u_{\mathbb{G},l};\mathbb{G}^*)$ and $\eta(E_{\mathbb{G},i} u;\mathbb{G}^*)$ for any $\mathbb{G}^*\in\mathbb{Z}^d$ are defined by
\begin{gather*}
	\eta(u_{\mathbb{G},l};\mathbb{G}^*)=\|\Pi_{\mathbb{G}^*}r(u_{\mathbb{G},l})\|_{H_p^{-1}},\quad\eta(E_{\mathbb{G},i} u;\mathbb{G}^*)=\|\Pi_{\mathbb{G}^*}r(E_{\mathbb{G},i} u)\|_{H_p^{-1}}.
\end{gather*}

For $U_\mathbb{G}=(u_{\mathbb{G},1},\dots,u_{\mathbb{G},N})\in (V_\mathbb{G})^N$, let
\begin{equation*}
	\eta^2(U_\mathbb{G};\mathbb{G}^*)=\sum_{l=1}^N\eta^2(u_{\mathbb{G},l};\mathbb{G}^*).
\end{equation*}
For $U=(U_1,\ldots,U_m)\in (H_p^1(\Omega))^N,\,U_i=(u_{k_i+1},\ldots,u_{k_i+q_i})\,(i=1,2,\ldots,m)$, let
\[
E_{\mathbb{G},i}U_i=(E_{\mathbb{G},i}u_{k_i+1},\ldots,E_{\mathbb{G},i}u_{k_i+q_i}),~\mathscr{E}_\mathbb{G} U=(E_{\mathbb{G},1}U_1,\ldots,E_{\mathbb{G},m}U_m)\in (V_\mathbb{G})^N
\]
and
\begin{equation*}
	\eta^2(\mathscr{E}_\mathbb{G} U;\mathbb{G}^*)=\sum_{i=1}^m\eta^2(E_{\mathbb{G},i} U_i;\mathbb{G}^*)=\sum_{i=1}^m\sum_{l=1}^{q_i}\eta^2(E_{\mathbb{G},i} u_{k_i+l};\mathbb{G}^*).
\end{equation*}

For simplicity, let $\eta(u_\mathbb{G})$, $\eta(U_\mathbb{G})$, $\eta(E_{\mathbb{G},i} U_i)$, and $\eta(\mathscr{E}_\mathbb{G} U)$ abbreviate $\eta(u_\mathbb{G};\mathbb{Z}^d)$, $\eta(U_\mathbb{G};\mathbb{Z}^d)$, $\eta(E_{\mathbb{G},i} U_i;\mathbb{Z}^d)$, and $\eta(\mathscr{E}_\mathbb{G} U;\mathbb{Z}^d)$, respectively.

\begin{theorem}\label{posteruelambdau}
	Let $\mathbb{G}\supset\mathbb{G}^{M_0}$. There exist constants $C_1$ and $C_2$, which only depend on the coercivity constant $\alpha_*$ and the continuity constant $\alpha^*$, such that 
	\begin{equation*}
		C_1\eta(E_{\mathbb{G},i} u)\leq
		\|u-E_{\mathbb{G},i} u\|_a\leq C_2\eta(E_{\mathbb{G},i} u)\quad \forall u\in M(\lambda_{(i)}),~i=1,2,\ldots,m
	\end{equation*}
	provided $M_0\gg 1$.
\end{theorem}

\begin{proof}
	Note that $Lw^{\mathbb{G},i}=\sum_{l=1}^{q_i}\alpha^\mathbb{G}_{i,l}(u) \lambda_{\mathbb{G},k_i+l}u_{\mathbb{G},k_i+l}$. It follows from \eqref{estimatbvp} that
	\begin{equation*}
		\frac{1}{\sqrt{a^*}}\eta(E_{\mathbb{G},i} u)\leq
		\|w^{\mathbb{G},i}-P_\mathbb{G} w^{\mathbb{G},i}\|_a\leq\frac{1}{\sqrt{a_*}}\eta(E_{\mathbb{G},i} u),
	\end{equation*}
	which together with \eqref{eigenbvp} completes the proof. In particular, we can choose the constants $C_1$ and $C_2$ satisfying
	\begin{equation*}
		C_1=\frac{1}{\sqrt{a^*}(1+\tilde{C}\gamma_{\mathbb{G}^{M_0}})},\quad C_2=\frac{1}{\sqrt{a_*}(1-\tilde{C}\gamma_{\mathbb{G}^{M_0}})},
	\end{equation*}
	where $\tilde{C}$ is introduced in \eqref{uminusw}.
\end{proof}

Motivated by \citet{dai2015convergence} and \citet{bonito2016convergence}, we obtain the following result.
\begin{lemma}\label{estimator-U_Lambda-E_Lambda U}
	Let $\mathbb{G}\supset \mathbb{G}^{M_0}$. Then, for any orthonormal basis $\{u_{k_i+l}\}_{l=1}^{q_i}$ of $M(\lambda_{(i)})~(i=1,2,\ldots,m)$, there holds
	\begin{equation*}
		\frac{1}{2}\eta^2(U_\mathbb{G})\leq\eta^2(\mathscr{E}_\mathbb{G} U)\leq \frac{3}{2}\eta^2(U_\mathbb{G}),
	\end{equation*}
	provided $M_0\gg 1$, where $U_\mathbb{G}=(u_{\mathbb{G},1},\ldots,u_{\mathbb{G},N})$ and
	\[
	\mathscr{E}_\mathbb{G} U=(E_{\mathbb{G},1} u_{k_1+1},\dots,E_{\mathbb{G},1}u_{k_1+q_1},\ldots,E_{\mathbb{G},m} u_{k_m+1},\dots,E_{\mathbb{G},m}u_{k_m+q_m}).
	\]
	Consequently,
	\begin{equation}\label{uandeu}
		\eta^2(U_\mathbb{G})\simeq
		\eta^2(\mathscr{E}_\mathbb{G} U).
	\end{equation}
\end{lemma}
\begin{proof}
	It follows from $M_\mathbb{G}(\lambda_{(i)})=\operatorname{span}\{u_{\mathbb{G},k_i+1},\ldots,u_{\mathbb{G},k_i+q_i}\}$ that there exist constants $\beta^{\mathbb{G},l}_{i,j}(j,l=1,\ldots,q_i)$ such that
	\begin{equation*}
		E_{\mathbb{G},i} u_{k_i+l}=\sum_{j=1}^{q_i} \beta^{\mathbb{G},l}_{i,j}u_{\mathbb{G},k_i+j},\quad l=1,2,\ldots,q_i.
	\end{equation*}
	Obviously $\beta^{\mathbb{G},l}_{i,j}=(E_{\mathbb{G},i} u_{k_i+l},u_{\mathbb{G},k_i+j}).$
	
	From the definition of $r(E_{\mathbb{G},i} u_{k_i+l})$, we get
	\begin{equation}\label{relaelambdau}
		\begin{aligned}
			r(E_{\mathbb{G},i} u_{k_i+l}) &=\sum_{j=1}^{q_i} \beta^{\mathbb{G},l}_{i,j}\lambda_{\mathbb{G},k_i+j}u_{\mathbb{G},k_i+j} +\Delta \left(\sum_{j=1}^{q_i} \beta^{\mathbb{G},l}_{i,j}u_{\mathbb{G},k_i+j}\right)- V\left(\sum_{j=1}^{q_i} \beta^{\mathbb{G},l}_{i,j}u_{\mathbb{G},k_i+j}\right)\\
			&=\sum_{j=1}^{q_i} \beta^{\mathbb{G},l}_{i,j}\left(\lambda_{\mathbb{G},k_i+j}u_{\mathbb{G},k_i+j}+\Delta u_{\mathbb{G},k_i+j}-V u_{\mathbb{G},k_i+j}\right)\\
			&=\sum_{j=1}^{q_i} \beta^{\mathbb{G},l}_{i,j} r(u_{\mathbb{G},k_i+j}),\quad i=1,2,\ldots,m,\, l=1,2,\ldots,q_i,
		\end{aligned}
	\end{equation}
	which indicates that $r(E_{\mathbb{G},i} u_{k_i+l})$ is a linear combination of  $r(u_{\mathbb{G},k_i+1}),\ldots,r(u_{\mathbb{G},k_i+q})$.
	
	We define the following matrices
	\[\tilde{V}_i\coloneqq[\tilde{v}_{k_i+l, \bm{G}}]\in\mathbb{R}^{q_i\times\infty},\quad V_i\coloneqq[v_{k_i+j,\bm{G}}]\in\mathbb{R}^{q_i\times\infty},\quad M_i\coloneqq [M_{k_i+l, k_i+j}]\in\mathbb{R}^{q_i\times q_i},
	\]
	where $\tilde{v}_{k_i+l, \bm{G}}=\hat{r}_{\bm{G}}(E_{\mathbb{G},i} u_{k_i+l})$, $v_{{k_i+j},\bm{G}}=\hat{r}_{\bm{G}}(u_{\mathbb{G},k_i+j})$, $M_{{k_i+l}, k_i+ j}=\beta^{\mathbb{G},l}_{i,j}$, $i=1,2,\ldots,m$, $j = 1,\ldots,q_i$, and $l=1,\ldots,q_i$. Thus \eqref{relaelambdau} can be written as
	\begin{equation*}
		\tilde{V}_i=M_i V_i,\quad i=1,2,\ldots,m.
	\end{equation*}
	Denote $\tilde{V}_{i,\bm{G}}$ by $[\tilde{v}_{k_i+1,\bm{G}},\ldots,\tilde{v}_{k_i+q_i,\bm{G}}]^T$ and
	$V_{i,\bm{G}}$ by $[v_{k_i+1,\bm{G}},\ldots,v_{k_i+q_i,\bm{G}}]^T$. Then for any $\bm{G}\in \mathbb{G}$, we have
	\begin{equation*}
		\tilde{V}_{i,\bm{G}}=M_i V_{i,\bm{G}},\quad i=1,2,\ldots,m.
	\end{equation*}
	Hence
	\begin{equation*}
		\|M_i^{-1}\|_2^{-1} \|V_{i,\bm{G}}\|_{2}\leq\|\tilde{V}_{i,\bm{G}}\|_{2}\leq\|M_i\|_2 \|V_{i,\bm{G}}\|_{2},
	\end{equation*}
	that is,
	\begin{equation*}
		\|M_i^{-1}\|_2^{-2} \sum_{j=1}^{q_i} |\hat{r}_{\bm{G}}(u_{\mathbb{G},k_i+j})|^2\leq\sum_{j=1}^{q_i} |\hat{r}_{\bm{G}}(E_{\mathbb{G},i} u_{k_i+j})|^2\leq\|M_i\|_2^2 \sum_{j=1}^{q_i} |\hat{r}_{\bm{G}}(u_{\mathbb{G},k_i+j})|^2.
	\end{equation*}
	Consequently,
	\begin{equation*}
		\|M_i^{-1}\|_2^{-2} \sum_{j=1}^{q_i} \frac{|\hat{r}_{\bm{G}}(u_{\mathbb{G},k_i+j})|^2}{1+|\bm{G}|^2}\leq\sum_{j=1}^{q_i} \frac{|\hat{r}_{\bm{G}}(E_{\mathbb{G},i} u_{k_i+j})|^2}{1+|\bm{G}|^2}\leq\|M_i\|_2^2 \sum_{j=1}^{q_i} \frac{|\hat{r}_{\bm{G}}(u_{\mathbb{G},k_i+j})|^2}{1+|\bm{G}|^2}.
	\end{equation*}
	Combining the definitions of $\eta^2(U_\mathbb{G})$ and $\eta^2(\mathscr{E}_\mathbb{G} U)$, we obtain
	\[
	\min_{i=1,\ldots,m}\|M_i^{-1}\|_2^{-2}\eta^2(U_\mathbb{G})\le \eta^2(\mathscr{E}_\mathbb{G} U)\le \max_{i=1,\ldots,m}\|M_i\|_2^2\eta^2(U_\mathbb{G}).
	\]
	
	Let $B_i = M_i M_i^*=[(E_{\mathbb{G},i} u_{k_i+l},E_{\mathbb{G},i} u_{k_i+j})]_{l,j=1}^{q_i}.$ We get from $$\|B_i\|_2=\|M_iM_i^*\|_2=\|M_i^*M_i\|_2=\|M_i\|_2^2$$
	that $\|M_i\|_2^2$  is equal to the largest eigenvalue of $B$.
	
	We see that Lemmas \ref{rholimits0} and \ref{deltalimits0} indicate $|\mathcal {O}(\rho_{\Omega}(\mathbb{G})\delta_\mathbb{G}(\lambda_{(i)}))|< \frac{1}{2q_i}$ provided $M_0\gg 1$. From Proposition \ref{elamdaiandj}, we get
	\begin{equation*}
		1-\frac{1}{2q_i}<1-|\mathcal {O}(\rho_{\Omega}(\mathbb{G})\delta_\mathbb{G}(\lambda_{(i)}))|\leq (B_i)_{jj}\leq 1+|\mathcal {O}(\rho_{\Omega}(\mathbb{G})\delta_\mathbb{G}(\lambda_{(i)}))|<1+\frac{1}{2q_i}
	\end{equation*}
	and
	\begin{equation*}
		\sum_{l\neq j}|(B_i)_{jl}|\leq (q_i-1)|\mathcal {O}(\rho_{\Omega}(\mathbb{G})\delta_\mathbb{G}(\lambda_{(i)}))|<\frac{q_i-1}{2q_i}.
	\end{equation*}
	Thus we obtain from Ger\v{s}gorin disc theorem that eigenvalues $\{\sigma_{i,j}\}_{j=1}^{q_i}$ of $B_i$ satisfy
	\begin{equation*}
		\frac{1}{2}=1- \frac{1}{2q_i}-\frac{q_i-1}{2q_i}\leq \sigma_{i,j}\leq 1+\frac{1}{2q_i}+\frac{q_i-1}{2q_i}=\frac{3}{2}, 1\leq j\leq q_i.
	\end{equation*}
	Therefore,
	\begin{equation*}
		\|M_i\|_2^2=\|B_i\|_2\leq \frac{3}{2},\quad i=1,2,\ldots,m,
	\end{equation*}
	which leads to
	\begin{equation*}
		\eta^2(\mathscr{E}_\mathbb{G} U)\leq \frac{3}{2}\eta^2(U_\mathbb{G}).
	\end{equation*}
	Similarly, 
	\begin{equation*}
		\|M_i^{-1}\|_2^2= \|(M_i^{-1})^*M_i^{-1}\|_2=\|B^{-1}\|_2=\frac{1}{\min_{1\leq j\leq q}\sigma_{i,j}}\leq 2,\quad i=1,2,\ldots,m,
	\end{equation*}
	which yields
	\begin{equation*}
		\frac{1}{2}\eta^2(U_\mathbb{G})\leq
		\eta^2(\mathscr{E}_\mathbb{G} U).
	\end{equation*}
\end{proof}

We turn to address the a posteriori error estimate for the distance between $\mathscr{M}$ and $\mathscr{M}_\mathbb{G}$.
\begin{theorem}
	Suppose $\mathbb{G}\supset\mathbb{G}^{M_0}$. Let $\lambda_{k_0+1}\le\lambda_{k_0+2}\le\cdots\le\lambda_{k_0+N}$ be a cluster of eigenvalues of \eqref{weakform}, which are $m$ eigenvalues $\lambda_{(1)}<\lambda_{(2)}<\cdots<\lambda_{(m)}$ with the corresponding eigenspace $M(\lambda_{(i)})$ if not accounting the multiplicity. Assume that the multiplicity of each eigenvalue is $q_i$ and satisfies $N=\sum_{i=1}^m q_i$. If $\mathscr{M}_{\mathbb{G}}$ is the planewave approximation of $\mathscr{M}$, then
	\begin{equation*}
		\delta_{H_p^1(\Omega)}(\mathscr{M},\mathscr{M}_{\mathbb{G}})\cong \eta(U_\mathbb{G})
	\end{equation*}
	provided $M_0\gg 1$, where $U_\mathbb{G}=(u_{\mathbb{G},1},\ldots,u_{\mathbb{G},N})$ and
	\[
	\delta^2_{H_p^1(\Omega)}(\mathscr{M},\mathscr{M}_{\mathbb{G}})=\sum_{i=1}^m \delta^2_{H_p^1(\Omega)}(M(\lambda_{(i)}),M_{\mathbb{G}}(\lambda_{(i)})).
	\]
\end{theorem}
\begin{proof}
	Let $\{u_{k_i+1},\ldots,u_{k_i+q_i}\}$ be an orthonormal basis of $M(\lambda_{(i)})$ and set $U_i= (u_{k_i+1},\ldots,u_{k_i+q_i})$ for $i=1,2,\ldots,m$. On the one hand,
	\begin{equation*}
		\begin{aligned}
			\sup\limits_{u\in M(\lambda_{(i)}),\|u\|_{a}=1}\inf\limits_{v\in M_\mathbb{G}(\lambda_{(i)})}\|u-v\|_{a}
			&\leq \sup\limits_{u\in M(\lambda_{(i)}),\|u\|_{a}=1}\|u-E_{\mathbb{G},i} u\|_{a}\\
			&\lesssim \max_{l=k_i+1,\ldots,k_i+q_i}\|u_l-E_{\mathbb{G},i} u_l\|_{a}\\
			&\lesssim \|U_i-E_{\mathbb{G},i} U_i\|_{a}.
		\end{aligned}
	\end{equation*}
	On the other hand,
	\begin{equation*}
		\begin{aligned}
			\sup\limits_{u\in M(\lambda_{(i)}),\|u\|_{a}=1}\inf\limits_{v\in M_\mathbb{G}(\lambda_{(i)})}\|u-v\|_{a}&\ge \sup\limits_{u\in M(\lambda_{(i)}),\|u\|_{a}=1}\|u-P_\mathbb{G} u\|_{a}\\
			&\gtrsim \max_{l=k_i+1,\ldots,k_i+q_i}\|u_l-P_\mathbb{G} u_l\|_{a}\\
			&\gtrsim \|U_i-P_\mathbb{G} U_i\|_{a}.
		\end{aligned}
	\end{equation*}
	We obtain from Theorem \ref{posteruelambdau} that
	\begin{equation*}
		\eta(E_{\mathbb{G},i} U_i)\lesssim\|U_i-E_{\mathbb{G},i} U_i\|_a\lesssim\eta(E_{\mathbb{G},i} U_i)
	\end{equation*}
	provided $M_0\gg 1$. Hence, from Proposition \ref{euandpu}, we have
	\begin{equation*}
		\eta(E_{\mathbb{G},i} U_i)\lesssim\|U_i-P_\mathbb{G} U_i\|_a
	\end{equation*}
	provided $M_0\gg 1$. Consequently,
	\begin{equation*}
		\eta(E_{\mathbb{G},i} U_i)\lesssim\sup\limits_{u\in M(\lambda_{(i)}),\|u\|_{a}=1}\inf\limits_{v\in M_\mathbb{G}(\lambda_{(i)})}\|u-v\|_{a}\lesssim \eta(E_{\mathbb{G},i} U_i),
	\end{equation*}
	namely,
	\[
	\eta(E_{\mathbb{G},i} U_i)\lesssim d_{H_p^1(\Omega)}(M(\lambda_{(i)}),M_{\mathbb{G}}(\lambda_{(i)}))\lesssim \eta(E_{\mathbb{G},i} U_i),
	\]
	which together with Lemma \ref{dxydyx} and $\dim M(\lambda_{(i)})=\dim M_\mathbb{G}(\lambda_{(i)})$ leads to
	\[
	\eta(E_{\mathbb{G},i} U_i)\lesssim \delta_{H_p^1(\Omega)}(M(\lambda_{(i)}),M_{\mathbb{G}}(\lambda_{(i)}))\lesssim \eta(E_{\mathbb{G},i} U_i).
	\]
	Combining \eqref{uandeu} and the above inequality, we arrive at
	\begin{equation*}
		\eta(U_\mathbb{G})\lesssim \delta_{H_p^1(\Omega)}(\mathscr{M},\mathscr{M}_{\mathbb{G}})\lesssim\eta(U_\mathbb{G}).
	\end{equation*}
\end{proof}

\subsection{Adaptive algorithm}
We now design an adaptive planewave algorithm with the D\"{o}rfler marking strategy for solving \eqref{weakform} as follows:
\begin{algorithm}[H]
	\caption{Adaptive planewave algorithm}
	\label{algorithmeigen}
	\begin{algorithmic}[1]
		\State Choose parameters $\theta\in(0,1)$, $tol\in[0,1)$ and $M_0\in\mathbb{N}$;
		\State Set the initial index set $\mathbb{G}_0=\mathbb{G}^{M_0}$ and $n=0$;
		\State\label{algostate:solveAPWM} Solve \eqref{discreform} with $\mathbb{G}$ being replaced by $\mathbb{G}_n$ to get the discrete solution $(\lambda_{\mathbb{G}_n,l},u_{\mathbb{G}_n,l})(l=1,\ldots,N)$;
		\State Compute the error estimator $\eta(U_{\mathbb{G}_{n}})$;
		\State If $\eta(U_{\mathbb{G}_{n}})<tol$, then stop;
		\State\label{algo-state:construct} Construct $\delta \mathbb{G}_n=\operatorname{\text{D\"ORFLER}}(\eta,U_{\mathbb{G}_n},\mathbb{G}_n,\theta)$;
		\State Set $\mathbb{G}_{n+1}=\mathbb{G}_n\cup\delta\mathbb{G}_n$;
		\State Let $n=n+1$ and go to Step \ref{algostate:solveAPWM}.
	\end{algorithmic}
\end{algorithm}

It follows from \eqref{discreform} that $\hat{r}_{\bm{G}}(u_{\mathbb{G},l})=0$ for any $l=1,2,\ldots,N$ and any $\bm{G}\in\mathbb{G}$. Thus, we have 
\begin{equation}\label{eq:deltaG-Gc}
	\eta(U_\mathbb{G};\delta \mathbb{G})=\eta(U_\mathbb{G};\mathbb{G}\cup\delta \mathbb{G}),
\end{equation}
which means that $\delta \mathbb{G}_n$ in \ref{algo-state:construct}-th step of Algorithm \ref{algorithmeigen} can always be constructed although $\delta \mathbb{G}_n\subset \mathbb{G}^c_n$.

We observe that $\eta(U_{\mathbb{G}})$ is a summation with infinite terms because of $\bm{G}\in\mathbb{Z}^d$, which means that $\eta(U_{\mathbb{G}})$ is uncomputable. As a result, we have to introduce a new estimator so that the adaptive algorithm can be executable. We choose an approximation $\tilde{r}(U_{\mathbb{G}})$ of $r(U_{\mathbb{G}})$ with finite Fourier expansions and expect that it holds for a given $\zeta\in(0,1)$ that
\begin{equation}\label{ineq:approx-r}
	\|\tilde{r}(U_{\mathbb{G}})-r(U_{\mathbb{G}})\|_{H_p^{-1}}\le\zeta\|\tilde{r}(U_{\mathbb{G}})\|_{H_p^{-1}}.
\end{equation}
Therefore, we easily get the following inequalities:
\[
(1-\zeta)\|\tilde{r}(U_{\mathbb{G}})\|_{H_p^{-1}}\le\|r(U_{\mathbb{G}})\|_{H_p^{-1}}\le(1+\zeta)\|\tilde{r}(U_{\mathbb{G}})\|_{H_p^{-1}}.
\]
We define a new error estimator as
\[
\tilde{\eta}(U_\mathbb{G})=\|\tilde{r}(U_\mathbb{G})\|_{H_p^{-1}},
\]
which is computable.

We then discuss how to choose $\tilde{r}$ so that \eqref{ineq:approx-r} is satisfied.
\begin{lemma}\label{lem:approx}
	Let $M_1$ and $M_2$ be two positive integers. Then
	\[
	\|Vu-(\Pi_{\mathbb{G}^{M_2}} V)u\|_{H_p^{-1}}\lesssim  \frac{3^d M_1^{d/6}}{M_2^{\sigma}}\|V\|_{H_p^\sigma}\|u\|_{L^2}\quad\forall u\in V_{\mathbb{G}^{M_1}}.
	\]
\end{lemma}
\begin{proof}
	By the definition of $H_p^{-1}(\Omega)$ norm and Nikolski's inequality \citep[see][3.4.3 (3)]{nikolskii1975approximation}, we have
	\begin{align*}
		\|Vu-(\Pi_{\mathbb{G}^{M_2}} V)u\|_{H_p^{-1}}&=\sup_{\|v\|_{H_p^1}=1}\langle Vu-(\Pi_{\mathbb{G}^{M_2}} V)u, v\rangle\\
		&\le\sup_{\|v\|_{H_p^1}=1}\|V-\Pi_{\mathbb{G}^{M_2}} V\|_{L^2}\|u\|_{L^3}\|v\|_{L^6}\\
		&\lesssim\sup_{\|v\|_{H_p^1}=1}\frac{1}{M_2^{\sigma}}\|V\|_{H_p^\sigma}3^d M_1^{d/6}\|u\|_{L^2}\|v\|_{H_p^1}\\
		&\lesssim \frac{3^d M_1^{d/6}}{M_2^{\sigma}}\|V\|_{H_p^\sigma}\|u\|_{L^2}.
	\end{align*}
\end{proof}

Let $\tilde{r}(u_{\mathbb{G},l})=\lambda_{\mathbb{G},l}u_{\mathbb{G},l}+\Delta u_{\mathbb{G},l}-(\Pi_{\mathbb{G}^{M}} V)u_{\mathbb{G},l}$. Then we get from Lemma \ref{lem:approx} that
\[
\|\tilde{r}(u_{\mathbb{G},l})-r(u_{\mathbb{G},l})\|_{H_p^{-1}}=\|Vu_{\mathbb{G},l}-(\Pi_{\mathbb{G}^M} V)u_{\mathbb{G},l}\|_{H_p^{-1}}\le C\frac{3^d \max_{\bm{G}\in\mathbb{G}}|\bm{G}|^{d/6}}{M^{\sigma}}\|V\|_{H_p^\sigma}.
\]
Thus we can increase $M$ until that $\displaystyle C\frac{3^d \max_{\bm{G}\in\mathbb{G}}|\bm{G}|^{d/6}}{M^{\sigma}}\|V\|_{H_p^\sigma}\le\zeta\|\tilde{r}(u_{\mathbb{G},l})\|_{H_p^{-1}}$, which derives \eqref{ineq:approx-r}. Namely, we have provided a possible implementation to build $\tilde{r}(U_\mathbb{G})$ satisfying \eqref{ineq:approx-r}.

The following lemma  is an extension of Lemma 3.1 in \citet{canuto2014adaptive} from $N=1$ to any $N$.
\begin{lemma}\label{lem:estimatorEqui}
	Let $\theta,\tilde{\theta}\in(0,1)$ and $\mathbb{G}^*\subset\mathbb{Z}^d$ be a finite index set. If
	\begin{equation*}
		\tilde{\eta}(U_\mathbb{G};\mathbb{G}^*)\coloneqq\|\Pi_{\mathbb{G}^*}\tilde{r}(U_\mathbb{G})\|_{H_p^{-1}}\ge\tilde{\theta}\tilde{\eta}(U_\mathbb{G})
	\end{equation*}
	and $\zeta\in(0,\tilde{\theta})$, then
	\begin{equation*}\label{eq:mark}
		\eta(U_\mathbb{G};\mathbb{G}^*)\ge \theta\eta(U_\mathbb{G}),\quad\text{with }\theta=\frac{\tilde{\theta}-\zeta}{1+\zeta}.
	\end{equation*}
	On the other hand, if $\eta(U_\mathbb{G};\mathbb{G}^*)\ge \theta\eta(U_\mathbb{G})$ and $\zeta\in(0,\theta/(1+\theta))$, then
	\begin{equation*}\label{eq:newmark}
		\tilde{\eta}(U_\mathbb{G};\mathbb{G}^*)\ge\tilde{\theta}\tilde{\eta}(U_\mathbb{G}),\quad\text{with }\tilde{\theta}=\theta-\zeta(1+\theta).
	\end{equation*}
\end{lemma}

Now we propose the following feasible adaptive planewave algorithm.
\begin{algorithm}[H]
	\caption{ Feasible adaptive planewave algorithm}
	\label{algo:FAPWM}
	\begin{algorithmic}[1]
		\State Choose parameters $\tilde{\theta}\in(0,1)$, $tol\in[0,1)$, $\zeta\in(0,\tilde{\theta})$ and $M_0\in\mathbb{N}$;
		\State Set the initial index set $\mathbb{G}_0=\mathbb{G}^{M_0}$ and $n=0$;
		\State\label{solveFAPWM} Solve \eqref{discreform} with $\mathbb{G}$ being replaced by $\mathbb{G}_n$ to get the discrete solution $(\lambda_{\mathbb{G}_n,l},u_{\mathbb{G}_n,l})(l=1,\ldots,N)$;
		\State Choose an approximation $\tilde{r}(U_{\mathbb{G}_n})$ of $r(U_{\mathbb{G}_n})$ such that
		\[
		\|\tilde{r}(U_{\mathbb{G}_n})-r(U_{\mathbb{G}_n})\|_{H_p^{-1}}\le\zeta\|\tilde{r}(U_{\mathbb{G}_n})\|_{H_p^{-1}};
		\]
		\State Compute the error estimator $\tilde{\eta}(U_{\mathbb{G}_{n}})$;
		\State If $\tilde{\eta}(U_{\mathbb{G}_{n}})<tol/(1+\zeta)$, then stop;
		\State Construct $\delta \mathbb{G}_n=\operatorname{\text{D\"ORFLER}}(\tilde{\eta},U_{\mathbb{G}_n},\mathbb{G}_n,\tilde{\theta})$;
		\State Set $\mathbb{G}_{n+1}=\mathbb{G}_n\cup\delta\mathbb{G}_n$;
		\State Let $n=n+1$ and go to Step \ref{solveFAPWM}.
	\end{algorithmic}
\end{algorithm}

\begin{lemma}\label{dolferre}
	Let $\mathbb{G}\supset\mathbb{G}^{M_0}$ and $\theta\in(0,1)$ be a given constant. If
	\begin{equation*}
		\eta(U_\mathbb{G};\mathbb{G}^*)\geq\theta\eta(U_\mathbb{G}),
	\end{equation*}
	then there exists a constant $\theta'\in (0,1)$, such that, for any orthonormal basis $\{u_{k_i+l}\}_{l=1}^{q_i}$ of $M(\lambda_{(i)})~(i=1,2,\ldots,m)$,
	\begin{equation*}
		\eta(\mathscr{E}_\mathbb{G} U;\mathbb{G}^*)\geq\theta'\eta(\mathscr{E}_\mathbb{G} U)
	\end{equation*}
	provided $M_0\gg 1$. Here $U=(u_1,\ldots,u_N)$.
\end{lemma}
\begin{proof}
	The arguments in the proof of Lemma \ref{estimator-U_Lambda-E_Lambda U} yield
	\begin{equation*}
		\frac{1}{2}\sum_{l=1}^{N}\|\Pi_{\mathbb{G}^*}r(u_{\mathbb{G},l})\|_{H_p^{-1}}^2\leq \sum_{i=1}^m\sum_{l=1}^{q_i}\|\Pi_{\mathbb{G}^*}r(E_{\mathbb{G},i} u_{k_i+l})\|_{H_p^{-1}}^2\leq \frac{3}{2}
		\sum_{l=1}^N\|\Pi_{\mathbb{G}^*}r(u_{\mathbb{G},l})\|_{H_p^{-1}}^2
	\end{equation*}
	when $M_0\gg 1$. Namely,
	\begin{equation*}
		\frac{1}{2}\eta^2(U_\mathbb{G};\mathbb{G}^*)\leq \eta^2(\mathscr{E}_\mathbb{G} U;\mathbb{G}^*)\leq \frac{3}{2}
		\eta^2(U_\mathbb{G};\mathbb{G}^*).
	\end{equation*}
	Therefore,
	\begin{equation*}
		\eta^2(\mathscr{E}_\mathbb{G} U;\mathbb{G}^*)\geq \frac{1}{2}
		\eta^2(U_\mathbb{G};\mathbb{G}^*)\geq \frac{1}{2}\theta^2
		\eta^2(U_\mathbb{G})\geq \frac{1}{3}\theta^2 \eta^2(\mathscr{E}_\mathbb{G} U).
	\end{equation*}
	Let $\theta'=\sqrt{\frac{1}{3}}\theta\in (0,1)$, we get
	\begin{equation*}
		\eta(\mathscr{E}_\mathbb{G} U;\mathbb{G}^*)\geq\theta'\eta(\mathscr{E}_\mathbb{G} U).
	\end{equation*}
\end{proof}

Similarly, we have
\begin{lemma}\label{dolferelamgdauandu}
	Let $\mathbb{G}\supset\mathbb{G}^{M_0}$ and $\theta\in(0,1)$ be a given constant. Let $\{u_{k_i+l}\}_{l=1}^{q_i}$ be an orthonormal basis of $M(\lambda_{(i)})~(i=1,2,\ldots,m)$. If
	\begin{equation*}
		\eta(\mathscr{E}_\mathbb{G} U;\mathbb{G}^*)\geq\theta\eta(\mathscr{E}_\mathbb{G} U),
	\end{equation*}
	where $U=(u_1,\ldots,u_N)$, then there holds
	\begin{equation*}
		\eta(U_\mathbb{G};\mathbb{G}^*)\geq\theta'\eta(U_\mathbb{G})
	\end{equation*}
	provided $M_0\gg 1$, where $\theta'=\sqrt{\dfrac{1}{3}}\theta\in (0,1)$.
\end{lemma}

\section{Convergence and complexity}\label{sec:conv-complex}
In this section, we analyze the asymptotic convergence and quasi-optimal complexity of the adaptive planewave method. The conclusions are valid for both Algorithms \ref{algorithmeigen} and \ref{algo:FAPWM}, although Algorithm \ref{algo:FAPWM} is stated in our analysis only.

\subsection{Convergence}
We shall first establish some relationships between two level planewave approximations.
\begin{lemma}
	Let $\mathbb{G}^{M_0}\subset\mathbb{G}_n\subset\mathbb{G}_{n+1}$ be subsets of $\mathbb{Z}^d$, $\{u_{k_i+l}\}_{l=1}^{q_i}$ be any orthonormal basis of $M(\lambda_{(i)})$ with $k_i=\sum_{j=1}^{i-1} q_i~(i=1,2,\ldots,m)$, $U=(u_1,u_2,\ldots,u_{k_m+q_m})$, $w^{\mathbb{G}_n,i}_j=\sum_{l=1}^{q_i}\alpha^{\mathbb{G}_n}_{i,l}(u_{k_i+j}) \lambda_{\mathbb{G}_n,k_i+l}Ku_{\mathbb{G}_n,k_i+l}(j=1,\ldots,q_i)$, $W^{\mathbb{G}_n,i}=(w^{\mathbb{G}_n,i}_1,\ldots,w^{\mathbb{G}_n,i}_{q_i})$, and $W^{\mathbb{G}_n}=(W^{\mathbb{G}_n,1},\ldots,W^{\mathbb{G}_n,m})$. Then,
	\begin{equation}\label{twolevelrela}
		\|U-\mathscr{E}_{\mathbb{G}_{n+1}} U\|_a=\|W^{\mathbb{G}_n}-P_{\mathbb{G}_{n+1}}W^{\mathbb{G}_n}\|_a+\mathcal {O}(\gamma_{\mathbb{G}_n})(\|U-\mathscr{E}_{\mathbb{G}_{n+1}} U\|_a+\|U-\mathscr{E}_{\mathbb{G}_{n}} U\|_a)
	\end{equation}
	provided $M_0\gg 1$.
\end{lemma}
\begin{proof}
	To get \eqref{twolevelrela}, it is sufficient to prove that, for any $u\in M(\lambda_{(i)})$ with $\|u\|_{L^2}=1$,
	\[
	w^{\mathbb{G}_n,i}=\sum_{l=1}^{q_i} \alpha^{\mathbb{G}_n}_{i,l}(u) \lambda_{\mathbb{G}_n,k_i+l}Ku_{\mathbb{G}_n,k_i+l}
	\]
	leads to
	\begin{equation*}
		\|u-E_{\mathbb{G}_{n+1},i} u\|_a=\|w^{\mathbb{G}_n,i}-P_{\mathbb{G}_{n+1}}w^{\mathbb{G}_n,i}\|_a+\mathcal {O}(\gamma_{\mathbb{G}_n})(\|u-E_{\mathbb{G}_{n+1},i} u\|_a+\|u-E_{\mathbb{G}_{n},i} u\|_a).
	\end{equation*}
	Let
	\[
	w^{\mathbb{G}_{n+1},i}=\sum_{l=1}^{q_i} \alpha^{\mathbb{G}_{n+1}}_{i,l}(u)\lambda_{\mathbb{G}_{n+1},k_i+l}Ku_{\mathbb{G}_{n+1},k_i+l}.
	\]
	It is clear that
	\begin{equation*}
		\begin{split}
			\|P_{\mathbb{G}_{n+1}}(w^{\mathbb{G}_{n},i}-w^{\mathbb{G}_{n+1},i})+u-w^{\mathbb{G}_{n},i}\|_a&\lesssim\|w^{\mathbb{G}_{n},i}-w^{\mathbb{G}_{n+1},i}\|_a+\|u-w^{\mathbb{G}_{n},i}\|_a\\
			&\lesssim\|u-w^{\mathbb{G}_{n+1},i}\|_a+\|u-w^{\mathbb{G}_{n},i}\|_a,
		\end{split}
	\end{equation*}
	which together with \eqref{uminusw} yields
	\begin{equation*}
		\begin{split}
			\|P_{\mathbb{G}_{n+1}}(w^{\mathbb{G}_{n},i}-w^{\mathbb{G}_{n+1},i})+u-w^{\mathbb{G}_{n},i}\|_a&\lesssim \gamma_{\mathbb{G}_{n+1}}\|u-E_{\mathbb{G}_{n+1},i} u\|_a+\gamma_{\mathbb{G}_{n}}\|u-E_{\mathbb{G}_{n},i} u\|_a\\
			&\lesssim \gamma_{\mathbb{G}_{n}}(\|u-E_{\mathbb{G}_{n+1},i} u\|_a+\|u-E_{\mathbb{G}_{n},i} u\|_a).
		\end{split}
	\end{equation*}
	Since \eqref{eq:Elambda_u} implies
	\begin{equation*}
		u-E_{\mathbb{G}_{n+1},i} u = u-w^{\mathbb{G}_{n},i}+P_{\mathbb{G}_{n+1}}(w^{\mathbb{G}_{n},i}-w^{\mathbb{G}_{n+1},i})+w^{\mathbb{G}_{n},i}-P_{\mathbb{G}_{n+1}}w^{\mathbb{G}_{n},i},
	\end{equation*}
	we get
	\begin{equation*}
		\|u-E_{\mathbb{G}_{n+1},i} u\|_a=\|w^{\mathbb{G}_n,i}-P_{\mathbb{G}_{n+1}}w^{\mathbb{G}_n,i}\|_a+\mathcal {O}(\gamma_{\mathbb{G}_n})(\|u-E_{\mathbb{G}_{n+1},i} u\|_a+\|u-E_{\mathbb{G}_{n},i} u\|_a).
	\end{equation*}
\end{proof}

Then we derive the error reduction.
\begin{theorem}\label{bvpeig}
	Let $\lambda_{k_0+1}\le\lambda_{k_0+2}\le\cdots\le\lambda_{k_0+N}$ be a cluster of eigenvalues of \eqref{weakform}, which are $m$ eigenvalues $\lambda_{(1)}<\lambda_{(2)}<\cdots<\lambda_{(m)}$ with the corresponding eigenspace $M(\lambda_{(i)})$ if not accounting the multiplicity. Assume that the multiplicity of each eigenvalue is $q_i$ and satisfies $N=\sum_{i=1}^m q_i$. Let $\{u_{k_i+l}\}_{l=1}^{q_i}$ be any orthonormal basis of $M(\lambda_{(i)})$ with $k_i=\sum_{j=1}^{i-1} q_i$ and $\{(\lambda_{\mathbb{G}_{n},l},u_{\mathbb{G}_{n},l})\in\mathbb{R}\times V_{\mathbb{G}_{n}}:l=1,\ldots,N\}_{n\in \mathbb{N}}$ be a sequence of planewave approximations produced by Algorithm \ref{algo:FAPWM}. Let $U=(u_1,\ldots,u_N)$. Then there exists a constant $\alpha\in(0,1)$, depending only on $\alpha^*$, $\alpha_*$, and parameters $\tilde{\theta},\zeta$ used in Algorithm \ref{algo:FAPWM}, such that, for any two consecutive iterations $n$ and $n+1$,
	\begin{equation}\label{ineq:error-reduction}
		\|U-\mathscr{E}_{\mathbb{G}_{n+1}} U\|_a\leq \alpha \|U-\mathscr{E}_{\mathbb{G}_{n}} U\|_a
	\end{equation}provided $M_0\gg 1$.
\end{theorem}
\begin{proof}
	It follows from Lemmas \ref{lem:estimatorEqui} and \ref{dolferre} that there exists a constant $\theta'\in (0,1)$ such that
	\begin{equation*}
		\eta(\mathscr{E}_\mathbb{G} U;\mathbb{G}^*)\geq\theta'\eta(\mathscr{E}_\mathbb{G} U)
	\end{equation*}
	provided $M_0\gg 1$. Recall that $w^{\mathbb{G}_n,i}_j=\sum_{l=1}^{q_i} \alpha^{\mathbb{G}_n}_{i,l}(u_{k_i+j}) \lambda_{\mathbb{G}_n,k_i+l}Ku_{\mathbb{G}_n,k_i+l}(j=1,\ldots,q_i)$, $W^{\mathbb{G}_n,i}=(w^{\mathbb{G}_n,i}_1,\ldots,w^{\mathbb{G}_n,i}_{q_i})$, and $W^{\mathbb{G}_n}=(W^{\mathbb{G}_n,1},\ldots,W^{\mathbb{G}_n,m})$. Thus D\"orfler marking strategy is satisfied with $\theta=\theta'$ for $W^{\mathbb{G}_n}$. We conclude from \eqref{convidealrho} and $P_{\mathbb{G}_n}W^{\mathbb{G}_n}=\mathscr{E}_{\mathbb{G}_n} U$ that there exists a constant $\rho\in(0,1)$ such that
	\begin{equation}\label{eq:sourceProblem-deduction}
		\|W^{\mathbb{G}_n}-P_{\mathbb{G}_{n+1}}W^{\mathbb{G}_n}\|_a\leq \rho\|W^{\mathbb{G}_n}-\mathscr{E}_{\mathbb{G}_n} U\|_a.
	\end{equation}
	
	Since $\gamma_{\mathbb{G}_n}$ decreases as $n$ increases, we obtain from \eqref{twolevelrela} and Young's inequality that
	\begin{align*}
		\|U-\mathscr{E}_{\mathbb{G}_{n+1}} U\|_a^2 &\leq (1+\delta)\|W^{\mathbb{G}_n}-P_{\mathbb{G}_{n+1}}W^{\mathbb{G}_n}\|^2_a\\
		&\quad+\hat{C}(1+\delta^{-1})\gamma_{\mathbb{G}_{0}}^2(\|U-\mathscr{E}_{\mathbb{G}_{n}} U\|_a^2+\|U-\mathscr{E}_{\mathbb{G}_{n+1}} U\|_a^2),
	\end{align*}
	where $\delta\in(0,1)$ is chosen to satisfy
	\begin{equation*}
		(1+\delta)\rho^2<1.
	\end{equation*}
	
	It follows from \eqref{uminusw} and \eqref{eq:sourceProblem-deduction} that
	\begin{equation*}
		\begin{split}
			&\mathrel{\phantom{=}}
			\|W^{\mathbb{G}_n}-P_{\mathbb{G}_{n+1}}W^{\mathbb{G}_n}\|^2_a\\
			&\leq\rho^2 \|W^{\mathbb{G}_n}-\mathscr{E}_{\mathbb{G}_n} U\|^2_a\\
			&\leq\rho^2(\|W^{\mathbb{G}_n}-U\|_a+\|U-\mathscr{E}_{\mathbb{G}_n} U\|_a)^2\\
			&\leq\rho^2(1+\tilde{C}\gamma_{\mathbb{G}_{0}})^2\|U-\mathscr{E}_{\mathbb{G}_n} U\|_a^2,
		\end{split}
	\end{equation*}
	which leads to
	\begin{align*}
		\|U-\mathscr{E}_{\mathbb{G}_{n+1}} U\|_a^2&\leq(1+\delta)\rho^2\|U-\mathscr{E}_{\mathbb{G}_n} U\|_a^2+C_4\delta^{-1}\gamma_{\mathbb{G}_{0}}(1+\gamma_{\mathbb{G}_{0}})\|U-\mathscr{E}_{\mathbb{G}_n} U\|_a^2\\
		&\quad+C_4\delta^{-1}\gamma^2_{\mathbb{G}_{0}}\|U-\mathscr{E}_{\mathbb{G}_{n+1}} U\|_a^2,
	\end{align*}
	where $C_4$ is some constant depending on $\tilde{C}$ and $\hat{C}$. Consequently,
	\begin{equation*}
		(1-C_4\delta^{-1}\gamma^2_{\mathbb{G}_{0}})\|U-\mathscr{E}_{\mathbb{G}_{n+1}} U\|_a^2
		\leq((1+\delta)\rho^2+C_4\delta^{-1}\gamma_{\mathbb{G}_{0}}(1+\gamma_{\mathbb{G}_{0}}))\|U-\mathscr{E}_{\mathbb{G}_{n}} U\|_a^2.
	\end{equation*}
	Let
	\begin{equation*}
		\alpha=\left(\frac{(1+\delta)\rho^2+C_4\delta^{-1}\gamma_{\mathbb{G}_{0}}}{1-C_4\delta^{-1}\gamma^2_{\mathbb{G}_0}}\right)^{\frac{1}{2}}
	\end{equation*}
	Since $M_0\gg 1$ implies $\gamma_{\mathbb{G}_0}\ll 1$, we have that $\alpha\in(0,1)$ and
	\begin{equation*}
		\|U-\mathscr{E}_{\mathbb{G}_{n+1}} U\|_a\leq \alpha \|U-\mathscr{E}_{\mathbb{G}_{n}} U\|_a
	\end{equation*}
	when $M_0\gg 1$.
\end{proof}

\begin{theorem}\label{decreasespace}
	Let $\lambda_{k_0+1}\le\lambda_{k_0+2}\le\cdots\le\lambda_{k_0+N}$ be a cluster of eigenvalues of \eqref{weakform}, which are $m$ eigenvalues $\lambda_{(1)}<\lambda_{(2)}<\cdots<\lambda_{(m)}$ with the corresponding eigenspace $M(\lambda_{(i)})$ if not accounting the multiplicity. Assume that the multiplicity of each eigenvalue is $q_i$ and satisfies $N=\sum_{i=1}^m q_i$. Let $\{(\lambda_{\mathbb{G}_{n},l},u_{\mathbb{G}_{n},l})\in\mathbb{R}\times V_{\mathbb{G}_{n}}:l=1,\ldots,N\}_{n\in \mathbb{N}}$ be a sequence of planewave approximations produced by Algorithm \ref{algo:FAPWM}. Set $M_{\mathbb{G}_n}(\lambda_{(i)})= \operatorname{span} \{u_{\mathbb{G}_{n},k_i+1},\ldots,u_{\mathbb{G}_{n},k_i+q_i}\}$ with $k_i=\sum_{j=1}^{i-1} q_j$. Then, there exists a constant $\alpha\in(0,1)$, depending only on $\alpha^*$, $\alpha_*$, and parameters $\tilde{\theta},\zeta$ used in Algorithm \ref{algo:FAPWM}, such that
	\begin{equation*}
		\delta_{H_p^1(\Omega)}(\mathscr{M},\mathscr{M}_{\mathbb{G}_n})\lesssim \alpha^n
	\end{equation*}
	provided $M_0\gg 1$.
\end{theorem}
\begin{proof}
	Let $\{u_{k_i+1},\ldots,u_{k_i+q_i}\}$ be an orthonormal basis of $M(\lambda_{(i)})$. For any $u\in M(\lambda_{(i)})$ with $\|u\|_{L^2}=1$, there exist $\{\alpha_{i,1},\ldots,\alpha_{i,q_i}\}$ such that $\sum_{l=1}^{q_i} |\alpha_{i,l}|^2 =1$ and $u=\sum_{l=1}^{q_i} \alpha_{i,l} u_{k_i+l}$.
	Therefore, we have
	\begin{equation*}
		\begin{aligned}
			\|u-E_{\mathbb{G}_n,i} u\|_a^2&=\left\|\sum_{l=1}^q \alpha_{i,l} (u_{k_i+l}-E_{\mathbb{G}_n,i} u_{k_i+l})\right\|_a^2\\
			&\leq\sum_{l=1}^{q_i} |\alpha_{i,l}|^2 \sum_{l=1}^{q_i}\|u_{k_i+l}-E_{\mathbb{G}_n,i} u_{k_i+l}\|_a^2
			=\sum_{l=1}^{q_i}\|u_{k_i+l}-E_{\mathbb{G}_n,i} u_{k_i+l}\|_a^2.\\
		\end{aligned}
	\end{equation*}
	We obtain from \eqref{ineq:error-reduction} and the definition of $d_{H_p^1(\Omega)}(M(\lambda),M_{\mathbb{G}_n}(\lambda))$ that
	\begin{align*}
		d_{H_p^1(\Omega)}(\mathscr{M},\mathscr{M}_{\mathbb{G}_n})&=\sum_{i=1}^m d^2_{H_p^1(\Omega)}(M(\lambda_{(i)}),M_{\mathbb{G}}(\lambda_{(i)}))\\
		&\le\sum_{i=1}^m \sum_{l=1}^{q_i}\|u_{k_i+l}-E_{\mathbb{G}_n,i} u_{k_i+l}\|_a^2=\|U-\mathscr{E}_{\mathbb{G}_n} U\|_a^2\\
		&\le\alpha^2\|U-\mathscr{E}_{\mathbb{G}_{n-1}}U\|_a^2\le\cdots\le\alpha^{2n} \|U-\mathscr{E}_{\mathbb{G}_0} U\|_a^2\lesssim\alpha^{2n},
	\end{align*}
	which together with Lemma \ref{dxydyx} implies
	\[
	d_{H_p^1(\Omega)}(\mathscr{M}_{\mathbb{G}_n},\mathscr{M})\lesssim \alpha^{2n}.
	\]
\end{proof}

\subsection{Complexity}
Following the complexity analysis in \citet{dai2008convergence,dai2015convergence}, we are able to analyze the complexity of the adaptive planewave approximations for eigenvalue problems by applying the complexity results for source problems. We first review some complexity results for the associated source system \eqref{weakformbvp}.

For the source system \eqref{weakformbvp}, we have the following local upper bound estimate and the D\"{o}rfler property, which are direct extensions of Lemmas 7.1 and 7.2 in \citet{canuto2014adaptive} from the case of $N=1$ to the case of any $N$.
\begin{proposition}
	Suppose $\mathbb{G}\subset\mathbb{G}^*\subset\mathbb{Z}^d$ be nonempty index sets. Let $U_{\mathbb{G}}$ and $U_{\mathbb{G}^*}$ be discrete solutions of \eqref{weakformbvp} in $(V_\mathbb{G})^N$ and $(V_{\mathbb{G}^*})^N$, respectively. Then the following local upper bound is valid
	\begin{equation*}
		\|U_{\mathbb{G}^*}-U_\mathbb{G}\|_a^2\leq
		\frac{1}{\alpha_*}\eta^2(U_\mathbb{G};\mathbb{G}^*).
	\end{equation*}
\end{proposition}

\begin{proposition}\label{dolfer}
	Suppose $\mathbb{G}\subset\mathbb{G}^*\subset\mathbb{Z}^d$ be nonempty index sets. Let $U_{\mathbb{G}}$ and $U_{\mathbb{G}^*}$  be discrete solutions of \eqref{weakformbvp} in $(V_\mathbb{G})^N$ and $(V_{\mathbb{G}^*})^N$, respectively. If
	\begin{equation*}
		\|U-U_{\mathbb{G}^*}\|_a^2\leq\mu\|U-U_\mathbb{G}\|_a^2
	\end{equation*}
	with $\mu\in(0,1)$, then $\mathbb{G}^*$ satisfies the D\"orfler condition, i.e.,
	\begin{equation*}
		\eta(U_\mathbb{G};\mathbb{G}^*)\geq\theta\eta(U_\mathbb{G}),
	\end{equation*}
	where $\theta=\sqrt{(1-\mu)\dfrac{\alpha_*}{\alpha^*}}$.
\end{proposition}

Next, we introduce a function approximation class as follows
\begin{equation*}
	\mathcal {A}^s:=\{v\in H_p^1(\Omega):|v|_{s}<\infty\},
\end{equation*}
where
\begin{equation*}
	|v|_{s}=\sup\limits_{\varepsilon>0}\varepsilon\inf\limits_{\{\mathbb{G}:\mathbb{G}_0\subset\mathbb{G}=-\mathbb{G}, \inf_{v_\mathbb{G}\in V_\mathbb{G}}\|v-v_\mathbb{G}\|_{a}\leq\varepsilon\}}(|\mathbb{G}|-|\mathbb{G}_0|)^s,
\end{equation*}
We see that $\mathcal{A}^s$ is the class of functions which can be approximated within a given tolerance $\varepsilon$ by trigonometric polynomials related to a index set $\mathbb{G}$ with number of degrees of freedom $|\mathbb{G}|-|\mathbb{G}_0|\leq\varepsilon^{-1/s}|v|_s^{1/s}$.

To derive the complexity of Algorithm \ref{algo:FAPWM}, we do some preparations.
\begin{lemma}\label{decrease}
	Let $\mathbb{G}^{M_0}\subset\mathbb{G}_n\subset\mathbb{G}_{n+1}$ be subsets of $\mathbb{Z}^d$, $\{u_{k_i+l}\}_{l=1}^{q_i}$ be any orthonormal basis of $M(\lambda_{(i)})$ with $k_i=\sum_{j=1}^{i-1} q_i~(i=1,2,\ldots,m)$, $U=(u_1,u_2,\ldots,u_{k_m+q_m})$, $w^{\mathbb{G}_n,i}_j=\sum_{l=1}^{q_i}\alpha^{\mathbb{G}_n}_{i,l}(u_{k_i+j}) \lambda_{\mathbb{G}_n,k_i+l}Ku_{\mathbb{G}_n,k_i+l}(j=1,\ldots,q_i)$, $W^{\mathbb{G}_n,i}=(w^{\mathbb{G}_n,i}_1,\ldots,w^{\mathbb{G}_n,i}_{q_i})$, and $W^{\mathbb{G}_n}=(W^{\mathbb{G}_n,1},\ldots,W^{\mathbb{G}_n,m})$. If the following property holds
	\begin{equation*}
		\|U-\mathscr{E}_{\mathbb{G}_{n+1}} U\|_{a}\leq \beta\|U-\mathscr{E}_{\mathbb{G}_{n}} U\|_{a}
	\end{equation*}
	with a constant $\beta\in(0,1)$, then for the associated source problems
	\[
	a(w^{\mathbb{G}_n,i}_j,v)=(\sum_{l=1}^{q_i}\alpha^{\mathbb{G}_n}_{i,l}(u_{k_i+j}) \lambda_{\mathbb{G}_n,k_i+l}u_{\mathbb{G}_n,k_i+l},v),~\forall v\in V_{\mathbb{G}_n}
	\]
	with $j=1,\ldots,q_i,\,i=1,\ldots,m$, we have
	\begin{equation*}
		\|W^{\mathbb{G}_n}-P_{\mathbb{G}_{n+1}}W^{\mathbb{G}_n}\|_a\leq \widetilde{\beta}\|W^{\mathbb{G}_n}-P_{\mathbb{G}_n}W^{\mathbb{G}_n}\|_a,
	\end{equation*}
	where
	\begin{equation*}
		\tilde{\beta}=\frac{[(1+\delta)\beta^2+\hat{C}(1+\delta^{-1})\gamma_{\mathbb{G}_0}^2(1+\beta^2)]^{\frac{1}{2}}}{1-\tilde{C}\gamma_{\mathbb{G}_0}}\in(0,1)
	\end{equation*}
	provided $M_0\gg 1$. Here $\delta\in (0,1)$ satisfies $(1+\delta)\beta^2<1$.
\end{lemma}
\begin{proof}
	It follows from \eqref{twolevelrela} and Young's inequality that
	\begin{align*}
		\|W^{\mathbb{G}_n}-P_{\mathbb{G}_{n+1}}W^{\mathbb{G}_n}\|_a^2&\leq (1+\delta)\|U-\mathscr{E}_{\mathbb{G}_{n+1}} U\|^2_a\\
		&\quad+\hat{C}(1+\delta^{-1})\gamma_{\mathbb{G}_0}^2
		(\|U-\mathscr{E}_{\mathbb{G}_{n}} U\|_a^2+\|U-\mathscr{E}_{\mathbb{G}_{n+1}} U\|_a^2),
	\end{align*}
	where $\delta\in(0,1)$ is chosen to satisfy
	\begin{equation*}
		(1+\delta)\beta^2<1.
	\end{equation*}
	By the assumption
	\begin{equation*}
		\|U-\mathscr{E}_{\mathbb{G}_{n+1}} U\|_{a}^2\leq \beta^2\|U-\mathscr{E}_{\mathbb{G}_{n}} U\|_{a}^2,
	\end{equation*}
	we obtain
	\begin{equation*}
		\|W^{\mathbb{G}_n}-P_{\mathbb{G}_{n+1}}W^{\mathbb{G}_n}\|_a^2\leq [(1+\delta)\beta^2+\hat{C}(1+\delta^{-1})\gamma_{\mathbb{G}_0}^2(1+\beta^2)]\|U-\mathscr{E}_{\mathbb{G}_{n}} U\|_{a,\Omega}^2,
	\end{equation*}
	Since \eqref{eigenbvp} implies
	\begin{equation*}
		\|U-\mathscr{E}_{\mathbb{G}_{n}} U\|_{a,\Omega}\leq\frac{1}{1-\tilde{C}\gamma_{\mathbb{G}_0}}\|W^{\mathbb{G}_n}-P_{\mathbb{G}_{n}}W^{\mathbb{G}_n}\|_a
	\end{equation*}
	provided $M_0\gg 1$, we have
	\begin{equation*}
		\|W^{\mathbb{G}_n}-P_{\mathbb{G}_{n+1}}W^{\mathbb{G}_n}\|_a^2\leq \frac{[(1+\delta)\beta^2+\hat{C}(1+\delta^{-1})\gamma_{\mathbb{G}_0}^2(1+\beta^2)]}
		{(1-\tilde{C}\gamma_{\mathbb{G}_0})^2}\|W^{\mathbb{G}_n}-P_{\mathbb{G}_{n}}W^{\mathbb{G}_n}\|_a^2.
	\end{equation*}
	Let
	\begin{equation}\label{debeta}
		\tilde{\beta}=\frac{[(1+\delta)\beta^2+\hat{C}(1+\delta^{-1})\gamma_{\mathbb{G}_0}^2(1+\beta^2)]^{\frac{1}{2}}}{1-\tilde{C}\gamma_{\mathbb{G}_0}},
	\end{equation}
	then $\tilde{\beta}\in(0,1)$ and
	\begin{equation*}
		\|W^{\mathbb{G}_n}-P_{\mathbb{G}_{n+1}}W^{\mathbb{G}_n}\|_a\leq \widetilde{\beta} \|W^{\mathbb{G}_n}-P_{\mathbb{G}_{n}}W^{\mathbb{G}_n}\|_a
	\end{equation*}
	provided $M_0\gg 1$.
\end{proof}

The following result is a direct consequence of $\mathscr{E}_{\mathbb{G}_n} U = P_{\mathbb{G}_n} W^{\mathbb{G}_n}$, Proposition \ref{dolfer}, and Lemma \ref{decrease}.
\begin{corollary}\label{eigendorfle}
	Let $\mathbb{G}^{M_0}\subset\mathbb{G}_n\subset\mathbb{G}_{n+1}$ be subsets of $\mathbb{Z}^d$, $\{u_{k_i+l}\}_{l=1}^{q_i}$ be any orthonormal basis of $M(\lambda_{(i)})$ with $k_i=\sum_{j=1}^{i-1} q_i~(i=1,2,\ldots,m)$, $U=(u_1,u_2,\ldots,u_{k_m+q_m})$. Suppose that the following decrease property holds
	\begin{equation*}
		\|U-\mathscr{E}_{\mathbb{G}_{n+1}} U\|_{a}\leq \beta\|U-\mathscr{E}_{\mathbb{G}_{n}} U\|_{a},
	\end{equation*}
	with a constant $\beta\in(0,1)$.
	Then $\mathbb{G}_{n+1}$ satisfies the D\"orfler condition, i.e.,
	\begin{equation*}
		\eta(\mathscr{E}_{\mathbb{G}_n }U;\mathbb{G}_{n+1})\geq\hat{\theta}\eta(\mathscr{E}_{\mathbb{G}_n} U)
	\end{equation*}
	with $\hat{\theta}=\sqrt{(1-\tilde{\beta})\frac{\alpha_*}{\alpha^*}}\in(0,1)$ provided $M_0\gg 1$, where $\tilde{\beta}$ is defined in \eqref{debeta}.
\end{corollary}

Different from the analysis of the convergence rate, additional requirements are needed for the analysis of the quasi-optimal complexity.

\begin{assumption}\label{asp:dorfler-minimal}
	$\delta \mathbb{G}_n=\operatorname{\textup{D\"ORFLER}}(\eta,U_{\mathbb{G}_n},\mathbb{G}_n,\theta)$ is of minimal cardinality for any $n\in\mathbb{N}$.
\end{assumption}

\begin{lemma}\label{lemma:diff-Lambda}
	Let $\lambda_{k_0+1}\le\lambda_{k_0+2}\le\cdots\le\lambda_{k_0+N}$ be a cluster of eigenvalues of \eqref{weakform}, which are $m$ eigenvalues $\lambda_{(1)}<\lambda_{(2)}<\cdots<\lambda_{(m)}$ with the corresponding eigenspace $M(\lambda_{(i)})$ if not accounting the multiplicity. Assume that the multiplicity of each eigenvalue is $q_i$ with $N=\sum_{i=1}^m q_i$ and $M(\lambda_{(i)})\subset\mathcal {A}^s$ for some $s$. Let $\{u_{k_i+l}\}_{l=1}^{q_i}$ be any orthonormal basis of $M(\lambda_{(i)})$ with $k_i=\sum_{j=1}^{i-1} q_i$ and $\{(\lambda_{\mathbb{G}_{n},l},u_{\mathbb{G}_{n},l})\in\mathbb{R}\times V_{\mathbb{G}_{n}}:l=1,\ldots,N\}_{n\in \mathbb{N}}$ be a sequence of planewave approximations produced by Algorithm \ref{algo:FAPWM} with the marking parameter $\tilde{\theta}\in\left(0,\sqrt{\dfrac{\alpha_*}{3\alpha^*}}\right)$ and $\displaystyle\zeta\in\left(0,\frac{\sqrt{\frac{\alpha_*}{3\alpha^*}}-\tilde{\theta}}{1+\sqrt{\frac{\alpha_*}{3\alpha^*}}} \right)$. Set $M_{\mathbb{G}_n}(\lambda_{(i)})= \operatorname{span} \{u_{\mathbb{G}_n,k_i+1},\ldots,u_{\mathbb{G}_{n},k_i+q_i}\}$. If Assumption \ref{asp:dorfler-minimal} is satisfied, then
	\begin{equation}\label{complexity}
		|\mathbb{G}_{n+1}|-|\mathbb{G}_n|\leq C\|U-\mathscr{E}_{\mathbb{G}_n} U\|_{a}^{-1/s}N^{1/2s}\sum_{l=1}^N|u_l|_s^{1/s}
	\end{equation}
	provided $M_0\gg 1$, where the constant $C$ depends on the discrepancy between $\zeta$ and $\displaystyle\frac{\sqrt{\frac{\alpha_*}{3\alpha^*}}-\tilde{\theta}}{1+\sqrt{\frac{\alpha_*}{3\alpha^*}}}$.
\end{lemma}
\begin{proof}
	Choose $\beta,\beta_1\in (0,1)$ satisfying $\beta_1\in(0,\frac12\beta)$ and
	\begin{equation}\label{eq:gamma}
		\zeta<\frac{\sqrt{(1-\beta)\frac{\alpha_*}{3\alpha^*}}-\tilde{\theta}}{1+\sqrt{(1-\beta)\frac{\alpha_*}{3\alpha^*}}},
	\end{equation}
	and set $\varepsilon=\beta_1\|U-\mathscr{E}_{\mathbb{G}_n}U\|_{a}$. Let $\delta_1\in(0,1)$ be some constant satisfying
	\begin{equation}\label{beta}
		(1+\delta_1)^2\beta_1^2< \frac14\beta^2,
	\end{equation}
	which indicates
	\begin{equation}\label{beta1}
		(1+\delta_1)\beta_1^2< \frac14.
	\end{equation}
	
	Let $\mathbb{G}_{\varepsilon,l}\supset\mathbb{G}_0$ with minimal cardinality satisfies $\mathbb{G}_{\varepsilon,l}=-\mathbb{G}_{\varepsilon,l}$ and 
	\begin{equation*}
		\|u_l-P_{\mathbb{G}_{\varepsilon,l}}u_l\|_{a}\leq\frac{\varepsilon}{\sqrt{N}},~l =1,\ldots,N,
	\end{equation*}
	which together with $\|u_l-P_{\mathbb{G}_{\varepsilon}}u_l\|_{a}\le\|u_l-P_{\mathbb{G}_{\varepsilon,l}}u_l\|_{a}$ leads to
	\begin{equation}\label{error1}
		\|U-P_{\mathbb{G}_\varepsilon}U\|_{a}\leq\varepsilon,
	\end{equation}
	where $\mathbb{G}_\varepsilon=\bigcup_{l=1}^N \mathbb{G}_{\varepsilon,l}$. Thus, by the definition of $\mathcal {A}^s$ and $\varepsilon$, we get that
	\begin{equation*}
		|\mathbb{G}_{\varepsilon,l}|-|\mathbb{G}_0|\leq \beta_1^{-1/s}\|U-\mathscr{E}_{\mathbb{G}_n}U\|_{a}^{-1/s}N^{1/2s}|u_l|_s^{1/s},~l=1,\ldots,N,
	\end{equation*}
	which implies
	\begin{equation*}
		|\mathbb{G}_{\varepsilon}|-|\mathbb{G}_0|\leq\sum_{l=1}^N(|\mathbb{G}_{\varepsilon,l}|-|\mathbb{G}_0|)\le \beta_1^{-1/s}\|U-\mathscr{E}_{\mathbb{G}_n}U\|_{a}^{-1/s} \left(N^{1/2s}\sum_{l=1}^N |u_l|_s^{1/s}\right).
	\end{equation*}
	Let $\mathbb{G}_{n,+}=\mathbb{G}_n\cup\mathbb{G}_\varepsilon$. It follows from $\mathbb{G}_0\subset\mathbb{G}_n\cap\mathbb{G}_\varepsilon$ that
	\begin{equation}\label{ineq:Gn+}
		|\mathbb{G}_{n,+}|-|\mathbb{G}_n|\leq |\mathbb{G}_\varepsilon|-|\mathbb{G}_0|.
	\end{equation}
	
	Let
	\begin{align*}
		w^{\mathbb{G}_\varepsilon,i}_l&=\sum_{j=1}^{q_i}\alpha^{\mathbb{G}_\varepsilon}_{i,j}(u_{k_i+l}) \lambda_{\mathbb{G}_\varepsilon,k_i+j}Ku_{\mathbb{G}_\varepsilon,k_i+j}\\ &=K\left(\sum_{j=1}^{q_i}\alpha^{\mathbb{G}_\varepsilon}_{i,j}(u_{k_i+l}) \lambda_{\mathbb{G}_\varepsilon,k_i+j}u_{\mathbb{G}_\varepsilon,k_i+j}\right), \quad l=1,\ldots,q_i,
	\end{align*}
	then it follows
	\begin{equation*}
		Lw^{\mathbb{G}_\varepsilon,i}_l=\sum_{j=1}^{q_i}\alpha^{\mathbb{G}_\varepsilon}_{i,j}(u_{k_i+l}) \lambda_{\mathbb{G}_\varepsilon,k_i+j}u_{\mathbb{G}_\varepsilon,k_i+j}.
	\end{equation*}
	Thus $\mathbb{G}_\varepsilon\subset\mathbb{G}_{n,+}$ yields
	\begin{equation*}
		\|W^{\mathbb{G}_\varepsilon}-P_{\mathbb{G}_{n,+}}W^{\mathbb{G}_\varepsilon}\|_{a}\leq \|W^{\mathbb{G}_\varepsilon}-P_{\mathbb{G}_{\varepsilon}}W^{\mathbb{G}_\varepsilon}\|_{a}.
	\end{equation*}
	We see from the proof of Theorem \ref{bvpeig} that
	\begin{equation}\label{error2}
		\|U-\mathscr{E}_{\mathbb{G}_{n,+}}U\|_{a}\leq \beta_0\|U-\mathscr{E}_{\mathbb{G}_{\varepsilon}}U\|_{a},
	\end{equation}
	where
	\begin{equation*}
		\beta_0=\left(\frac{(1+\delta_1)+C_4\delta_1^{-1}\gamma_{\mathbb{G}_0}}
		{1-C_4\delta_1^{-1}\gamma^2_{\mathbb{G}_0}}\right)^{\frac{1}{2}}
	\end{equation*}
	and $C_4$ is the constant in the proof of Theorem \ref{bvpeig}.
	Combining \eqref{error1}, \eqref{error2}, with Proposition \ref{euandpu}, we obtain
	\begin{equation*}
		\|U-\mathscr{E}_{\mathbb{G}_{n,+}}U\|_{a}\leq \check{\beta}\|U-\mathscr{E}_{\mathbb{G}_{n}}U\|_{a}
	\end{equation*}
	provided $M_0\gg 1$, where $\check{\beta}=2\beta_0\beta_1$. Note that \eqref{beta1} and $M_0\gg 1$ imply $\check{\beta}\in(0,1)$. Thus It follows from Corollary \ref{eigendorfle} that
	\begin{equation*}
		\eta(\mathscr{E}_{\mathbb{G}_{n}}U;\mathbb{G}_{n,+})\geq\check{\theta}\eta(\mathscr{E}_{\mathbb{G}_{n}}U),
	\end{equation*}
	where $\check{\theta}=\sqrt{(1-\hat{\beta})\dfrac{\alpha_*}{\alpha^*}}, \hat{\beta}=\dfrac{[(1+\delta_1)\check{\beta}^2+\hat{C}(1+\delta_1^{-1})\gamma_{\mathbb{G}_0}^2(1+\check{\beta}^2)]^{\frac{1}{2}}}{1-\tilde{C}\gamma_{\mathbb{G}_0}}$.
	
	We conclude from \eqref{eq:deltaG-Gc}, Lemmas \ref{lem:estimatorEqui} and \ref{dolferelamgdauandu} that
	\begin{equation*}
		\tilde{\eta}(U_{\mathbb{G}_{n}};\mathbb{G}_{n,+}\setminus \mathbb{G}_n)\geq\check{\theta}'\tilde{\eta}(U_{\mathbb{G}_{n}}),
	\end{equation*}
	where $\check{\theta}'=\sqrt{\dfrac13}\check{\theta}-\zeta\left(1+\sqrt{\dfrac13}\check{\theta}\right)=\sqrt{(1-\hat{\beta})\dfrac{\alpha_*}{3\alpha^*}}-\zeta\left(1+\sqrt{(1-\hat{\beta})\dfrac{\alpha_*}{3\alpha^*}}\right)$. It is observed from $M_0\gg 1$ and \eqref{beta} that $\hat{\beta}\in(0,\beta)$, which together with \eqref{eq:gamma} leads to $\check{\theta}'>\tilde{\theta}$.
	Since $\delta\mathbb{G}_n$ constructed by D\"orfler marking strategy satisfies
	\begin{equation*}
		\tilde{\eta}(U_{\mathbb{G}_n};\delta\mathbb{G}_n)\geq\tilde{\theta}\eta(U_{\mathbb{G}_n}),
	\end{equation*} 
	with minimal cardinality and $\delta\mathbb{G}_n\subset\mathbb{G}_n^c$, we obtain from \eqref{ineq:Gn+} that
	\begin{equation*}
		\begin{split}
			|\mathbb{G}_{n+1}|-|\mathbb{G}_n|&=|\delta\mathbb{G}_n|\le|\mathbb{G}_{n,+}\setminus \mathbb{G}_n|=|\mathbb{G}_{n,+}|-|\mathbb{G}_n|\leq |\mathbb{G}_\varepsilon|-|\mathbb{G}_0|\\
			&\leq\beta_1^{-1/s}\|U-\mathscr{E}_{\mathbb{G}_n}U\|_{a}^{-1/s} \left(N^{1/2s}\sum_{l=1}^N|u_l|_s^{1/s}\right),
		\end{split}
	\end{equation*}
	which is \eqref{complexity} with an explicit dependence on the discrepancy between $\zeta$ and $\displaystyle\frac{\sqrt{\frac{\alpha_*}{3\alpha^*}}-\tilde{\theta}}{1+\sqrt{\frac{\alpha_*}{3\alpha^*}}}$.
\end{proof}

Finally, we obtain that Algorithm \ref{algo:FAPWM} possesses quasi-optimal complexity.
\begin{theorem}\label{complexitytheorem}
	Let $\lambda_{k_0+1}\le\lambda_{k_0+2}\le\cdots\le\lambda_{k_0+N}$ be a cluster of eigenvalues of \eqref{weakform}, which are $m$ eigenvalues $\lambda_{(1)}<\lambda_{(2)}<\cdots<\lambda_{(m)}$ with the corresponding eigenspace $M(\lambda_{(i)})$ if not accounting the multiplicity. Assume that the multiplicity of each eigenvalue is $q_i$ with $N=\sum_{i=1}^m q_i$ and $M(\lambda_{(i)})\subset\mathcal {A}^s$ for some $s$. Let $\{u_{k_i+l}\}_{l=1}^{q_i}$ be any orthonormal basis of $M(\lambda_{(i)})$ with $k_i=\sum_{j=1}^{i-1} q_i$ and $\{(\lambda_{\mathbb{G}_{n},l},u_{\mathbb{G}_{n},l})\in\mathbb{R}\times V_{\mathbb{G}_{n}}:l=1,\ldots,N\}_{n\in \mathbb{N}}$ is a sequence of planewave approximations produced by Algorithm \ref{algo:FAPWM} with marking parameter $\tilde{\theta}\in\left(0,\sqrt{\dfrac{\alpha_*}{3\alpha^*}}\right)$ and $\displaystyle\zeta\in\left(0,\frac{\sqrt{\frac{\alpha_*}{3\alpha^*}}-\tilde{\theta}}{1+\sqrt{\frac{\alpha_*}{3\alpha^*}}} \right)$. Set $M_{\mathbb{G}_n}(\lambda_{(i)})= \operatorname{span} \{u_{\mathbb{G}_n,k_i+1},\ldots,u_{\mathbb{G}_{n},k_i+q_i}\}$.  If Assumption \ref{asp:dorfler-minimal} is satisfied, then the $n$-th iteration solution space $\mathscr{M}_{\mathbb{G}_n}$ satisfies
	\begin{equation}\label{ineq:E U-opt-bound}
		\|U-\mathscr{E}_{\mathbb{G}_n} U\|_{a}\lesssim (|\mathbb{G}_n|-|\mathbb{G}_0|)^{-s},
	\end{equation}
	\begin{equation}\label{ineq:lambda-opt-bound}
		\lambda_{\mathbb{G}_{n},k_i+l}-\lambda_{(i)}\lesssim (|\mathbb{G}_n|-|\mathbb{G}_0|)^{-2s},\quad l=1,\ldots,q_i,~i=1,\ldots,m
	\end{equation}
	provided $M_0\gg 1$.
\end{theorem}
\begin{proof}
	We obtain from Theorem \ref{bvpeig} that there exists $\alpha\in (0,1)$ such that it holds for any $0\leq j<n$ that
	\begin{equation*}
		\|U-\mathscr{E}_{\mathbb{G}_j} U\|_{a}^{-1/s}\leq \alpha^{(n-j)/s} \|U-\mathscr{E}_{\mathbb{G}_n} U\|_{a}^{-1/s},
	\end{equation*}
	which together with Lemma \ref{lemma:diff-Lambda} leads to
	\begin{equation*}
		\begin{aligned}
			|\mathbb{G}_n|-|\mathbb{G}_0|&=\sum^{n-1}_{j=0}(|\mathbb{G}_{j+1}|-|\mathbb{G}_j|)\\
			&\lesssim\left(\sum_{l=1}^N |u_l|_s^{1/s}\right)\sum^{n-1}_{j=0}\|U-\mathscr{E}_{\mathbb{G}_j}U\|_{a}^{-1/s}\\
			&\lesssim \|U-\mathscr{E}_{\mathbb{G}_n}U\|_{a}^{-1/s}\left(\sum_{l=1}^N |u_l|_s^{1/s}\right)\sum^{n}_{j=1}\alpha^{j/s}.
		\end{aligned}
	\end{equation*}
	Namely,
	\begin{equation*}
		|\mathbb{G}_n|-|\mathbb{G}_0|\lesssim \|U-\mathscr{E}_{\mathbb{G}_n}U\|_{a}^{-1/s}\left(\sum_{l=1}^N |u_l|_s^{1/s}\right)
	\end{equation*}
	because of $\alpha<1$, which gives \eqref{ineq:E U-opt-bound}. On the other hand, combining \eqref{lambdaineq}, the definition of $\delta_\mathbb{G}(\lambda)$ with \eqref{ineq:E U-opt-bound}, we arrive at \eqref{ineq:lambda-opt-bound}.
\end{proof}

Similar to Thorem \ref{decreasespace}, we obtain the following conclusion from Theorem \ref{complexitytheorem}.
\begin{theorem}\label{thm:opt-complex}
	Let $\lambda_{k_0+1}\le\lambda_{k_0+2}\le\cdots\le\lambda_{k_0+N}$ be a cluster of eigenvalues of \eqref{weakform}, which are $m$ eigenvalues $\lambda_{(1)}<\lambda_{(2)}<\cdots<\lambda_{(m)}$ with the corresponding eigenspace $M(\lambda_{(i)})$ if not accounting the multiplicity. Assume that the multiplicity of each eigenvalue is $q_i$ with $N=\sum_{i=1}^m q_i$ and $M(\lambda_{(i)})\subset\mathcal {A}^s$ for some $s$. Let $\{u_{k_i+l}\}_{l=1}^{q_i}$ be any orthonormal basis of $M(\lambda_{(i)})$ with $k_i=\sum_{j=1}^{i-1} q_i$ and $\{(\lambda_{\mathbb{G}_{n},l},u_{\mathbb{G}_{n},l})\in\mathbb{R}\times V_{\mathbb{G}_{n}}:l=1,\ldots,N\}_{n\in \mathbb{N}}$ is a sequence of planewave approximations produced by Algorithm \ref{algo:FAPWM} with marking parameter $\tilde{\theta}\in\left(0,\sqrt{\dfrac{\alpha_*}{3\alpha^*}}\right)$ and $\displaystyle\zeta\in\left(0,\frac{\sqrt{\frac{\alpha_*}{3\alpha^*}}-\tilde{\theta}}{1+\sqrt{\frac{\alpha_*}{3\alpha^*}}} \right)$. Set $M_{\mathbb{G}_n}(\lambda_{(i)})= \operatorname{span} \{u_{\mathbb{G}_n,k_i+1},\ldots,u_{\mathbb{G}_{n},k_i+q_i}\}$.  If Assumption \ref{asp:dorfler-minimal} is satisfied, then the $n$-th iterate solution space $\mathscr{M}_{\mathbb{G}_n}$ satisfies
	\begin{equation*}
		\delta_{H_p^1(\Omega)}(\mathscr{M},\mathscr{M}_{\mathbb{G}_n}) \lesssim (|\mathbb{G}_n|-|\mathbb{G}_0|)^{-s}
	\end{equation*}
	provided $M_0\gg 1$.
\end{theorem}

\section{Concluding remarks}\label{sec:concl}
In this paper, we have  designed and analyzed an adaptive planewave algorithm for a class of second-order elliptic eigenvalue problems. Following \citet{dai2008convergence,dai2015convergence}, we apply the relationship between the planewave approximation for the elliptic eigenvalue problem and those for the corresponding source problems to carry out the analysis for the convergence rate and the complexity of the adaptive planewave algorithm. We have proved that the adaptive planewave algorithm has the asymptotic contraction property and quasi-optimal complexity. We would like to mention that the current work is a theoretical part of the adaptive planewave method for eigenvalue problems. To implement the adaptive planewave method, there are a couple of practical issues that need to be addressed, which together with applications to electronic structure calculations is our on-going work and will be addressed elsewhere.


\section*{Funding}
This work was supported by the National Key R \& D Program of China under grants 2019YFA0709600 and 2019YFA0709601, the National Natural Science Foundation of China under grants 12021001 and 11671389, and the National Center for Mathematics and Interdisciplinary Sciences, CAS.

\bibliographystyle{abbrvnat}
\bibliography{refs}

\end{document}